\documentclass[reqno, 11pt]{amsart}
\usepackage[letterpaper, margin=1in]{geometry}
\usepackage{amsthm, amsmath, amsfonts, amssymb}
\usepackage{enumerate, mathabx}
\usepackage{xcolor}
\usepackage{hyperref}

\usepackage{scalerel,stackengine}
\stackMath
\newcommand\reallywidecheck[1]{%
\savestack{\tmpbox}{\stretchto{%
  \scaleto{%
    \scalerel*[\widthof{\ensuremath{#1}}]{\kern-.6pt\bigwedge\kern-.6pt}%
    {\rule[-\textheight/2]{1ex}{\textheight}}%WIDTH-LIMITED BIG WEDGE
  }{\textheight}% 
}{0.5ex}}%
\stackon[1pt]{#1}{\scalebox{-1}{\tmpbox}}%
}

\newsavebox{\accentbox}
\newcommand{\compositeaccents}[2]{%
  \sbox\accentbox{$#2$}#1{\usebox\accentbox}}

\newcommand{\gkbbh}[0]{\compositeaccents{\widehat}{\overline{g_{\bar{K}_1}}} \ast \cdots \ast \compositeaccents{\widehat}{\overline{g_{\bar{K}_k}}}}

\def\wb1{w_{B_{\delta^{-1}}}}

\numberwithin{equation}{section}

\newcommand{\F}[0]{\mathbb{Q}_q}
\renewcommand{\O}[0]{\mathbb{Z}_q}

\newcommand{\Z}[0]{\mathbb{Z}}
\newcommand{\R}[0]{\mathbb{R}}
\newcommand{\Q}[0]{\mathbb{Q}}
\newcommand{\C}[0]{\mathbb{C}}
\newcommand{\N}[0]{\mathbb{N}}
\newcommand{\T}[0]{\mathbb{T}}
\newcommand{\Om}[0]{\Omega}
\newcommand{\ta}[0]{\theta}

\newcommand{\vep}[0]{\varepsilon}
\newcommand{\supp}[0]{\operatorname{supp}}
\newcommand{\lsm}[0]{\lesssim}

\newcommand{\wh}[1]{\widehat{#1}}

\newcommand{\st}[1]{\substack{#1}}

\newcommand{\mf}[1]{\mathfrak{#1}}

\newcommand{\nms}[1]{\| #1 \|}

\newcommand{\Mod}[1]{\ (\mathrm{mod}\ #1)}

\definecolor{orange}{rgb}{1,0.5,0}

\newtheorem{thm}{Theorem}[section]
\newtheorem{lemma}[thm]{Lemma}
\newtheorem{prop}[thm]{Proposition}
\newtheorem{cor}[thm]{Corollary}

\theoremstyle{remark}

\title[A decoupling interpretation of an old argument for VMVT]{A decoupling interpretation of an old argument for Vinogradov's Mean Value Theorem}

\author[B. Cook]{Brian Cook} 
\address{Department of Mathematics, Virginia Tech, Blacksburg, VA 24061-0123, USA}
\email{briancookmath@gmail.com}

\author[K. Hughes]{Kevin Hughes}
\address{School of Mathematics, The University of Bristol, Bristol BS8 1UG; and the Heilbronn Insitute for Mathematical Research, Bristol, UK}
\email{khughes.math@gmail.com} 

\author[Z. K. Li]{Zane Kun Li}
\address{Department of Mathematics, North Carolina State University, Raleigh, NC 27695, USA}
\email{zkli@ncsu.edu}

\author[A. Mudgal]{Akshat Mudgal}
\address{Mathematical Institute, University of Oxford, Oxford OX2 6GG, UK}
\email{Akshat.Mudgal@maths.ox.ac.uk}

\author[O. Robert]{Olivier Robert}
\address{Universit\'e de Lyon,  Universit\'e de Saint-\'Etienne, CNRS UMR 5208, Institut Camille Jordan,   F-42000 Saint-\'Etienne, France.}
\email{olivier.robert@univ-st-etienne.fr}

\author[P.-L. Yung]{Po-Lam Yung}
\address{Mathematical Sciences Institute, Australian National University, Canberra ACT 2601, Australia}
\email{PoLam.Yung@anu.edu.au}

\begin{document}
\begin{abstract}
We interpret into decoupling language a refinement of a 1973 argument due to Karatsuba on Vinogradov's mean value theorem. The main goal of our argument is to answer what precisely solution counting in older partial progress on Vinogradov's mean value theorem corresponds to in Fourier decoupling theory.
\end{abstract}
\maketitle
\section{Introduction}
\subsection{Motivation}
Let $s \geq 1$ and $k \geq 2$ be integers. For $X \geq 1$,
let $J_{s, k}(X)$ be the number of solutions to the degree $k$
Vinogradov system in $2s$ variables:
\begin{align} \label{eq:Vinosystem}
x_1^j + x_{2}^{j} + \dots + x_s^j &= y_1^j + y_{2}^{j} + \dots + y_s^j, \quad 1 \leq j \leq k
\end{align}
where all variables $x_1, \ldots, x_s, y_1, \ldots, y_s \in [1, X] \cap \N$.
Nontrivial upper bounds for $J_{s, k}(X)$ were first studied by Vinogradov in 1935 \cite{Vinogradov1935}
and such results are collectively referred to as Vinogradov's Mean Value Theorem (VMVT)
in the literature. The main conjecture in VMVT, now a theorem as of 2015, was that for every $\vep > 0$ and $s,k \in \mathbb{N}$, one has
\begin{equation} \label{eq:VMT}
J_{s,k}(X) \lsm_{s, k, \vep} X^{\varepsilon} (X^s + X^{2s-\frac{k(k+1)}{2}})
\end{equation}
for all $X \geq 1$.
It is not hard to see that $J_{s, k}(X) \gtrsim_{s, k} X^{s} + X^{2s - k(k + 1)/2}$
and applying H\"{o}lder's inequality, we may deduce \eqref{eq:VMT} for all $s \in \mathbb{N}$ from the $s = k(k+1)/2$ case. VMVT plays an important role in understanding Waring's problem and the Riemann zeta function, see for example \cite{F02-zeta, F02-zerofree, HB17, WooleyICM}.
When $k = 2$, the main conjecture in VMVT is classical. In 2014, Wooley \cite{W16} proved the $k = 3$ case of VMVT
using the method of efficient congruencing (see also \cite{HB15} for a shorter proof due to Heath-Brown). In 2015, the $k \geq 2$
case was proven by Bourgain, Demeter, and Guth in \cite{BDG} using Fourier decoupling for the degree $k$ moment curve from which VMVT followed as a corollary.
Finally, in 2017, Wooley \cite{W19}, gave an alternative proof of \eqref{eq:VMT} for all $k \geq 2$ using nested efficient congruencing.

After the proofs of VMVT using the Fourier method of decoupling \cite{BDG} and the number theoretic method of efficient congruencing \cite{W19},
it has been an interesting question to determine how these two methods are related and 
whether a ``dictionary" between the two methods could be obtained. The study of this dictionary has
led to new proofs of Fourier decoupling for the parabola \cite{Li18}, cubic moment curve \cite{GLY19}, and the degree $k$ moment curve \cite{GLYZ}; these having been inspired from the efficient congruencing arguments in \cite[Section 4]{Pierce19}, \cite{HB15}, and \cite{W19}, respectively. Additionally,
a decoupling interpretation of the study of VMVT over ellipspephic sets \cite{Biggs19} led to a proof of Fourier decoupling for fractal sets on the parabola \cite{CDGJLM}.

In this article, we revisit a particular classical VMVT which states that
\begin{equation} \label{eq:VMTold}
J_{s,k}(X) \lsm_{s, k} X^{2s-\frac{k(k+1)}{2} + \frac{1}{2}k^2(1-\frac{1}{k})^{s/k}}
\end{equation}
for all $X \geq 1$ and $s = kl$ with $l \in \N$. 
This result should be compared to the supercritical $s \geq k(k + 1)/2$ case in \eqref{eq:VMT}. For $s$ very large compared to $k$, we have an extra term $\frac{1}{2}k^2(1-\frac{1}{k})^{s/k}$ in the exponent, which decays exponentially in $s$ for every fixed value of $k$, instead of an $\vep$.
The estimate \eqref{eq:VMTold} appears (for example) in Vaughan's book \cite[Chapter 5]{V} and is
a refinement of an argument of Karatsuba \cite{Karatsuba} from 1973 (see also Stechkin \cite{Stechkin} from 1975).
The loss of the $X^{\frac{1}{2}k^2(1 - \frac{1}{k})^{s/k}}$ comes from combining the subcritical estimate $J_{k, k}(X) \lsm_{k} X^k$, which follows from the Newton-Girard identities, along with an iterative argument to derive estimates for $J_{s, k}(X)$
when $s$ is supercritical.

The main purpose of this paper
is to illustrate how this refined argument of Karatsuba can be adapted to give a proof
of a non-sharp Fourier decoupling inequality for the degree $k$ moment curve in the supercritical regime.
The key difficulty that prevents the direct use of ideas from \cite{GLY19, GLYZ, Li18}
is the heavy reliance on solution counting in \eqref{eq:VMTold}.
One of the main points of this article is to clarify the role of such solution counting arguments 
in the study of Fourier decoupling. The mechanism driving the solution counting arguments will allow 
us to prove the key Lemma~\ref{lem:counting} below, which concerns the geometry of Fourier supports of the functions appearing in our main Theorem~\ref{thm:main}.

Since our goal is to clarify the role of solution counting in Fourier decoupling and Bourgain, Demeter, and Guth have
already given the sharpest possible moment curve decoupling theorem in \cite{BDG}, we will work over $\F$ rather than over $\R$.
This will allow us
to present the argument in the cleanest possible manner, free of technical difficulties arising from the inconvenience of the uncertainty principle in $\R^k$.
See also \cite{GLY21} for another decoupling paper that works over $\F$ rather than $\R$, there however, the authors use the observation that decoupling over $\F$ is quantitatively more efficient than decoupling over $\R$ in terms of exponential sum estimates.

\subsection*{Notation} 
As $k$ will be fixed, we will allow all constants to depend on $k$.
Given two positive expressions $X$ and $Y$, we write
$X \lsm Y$ if $X \leq CY$ for some constant $C$ that is allowed to depend on $k$. If $C$ depends on some additional parameter $A$,
then we write $X \lsm_{A} Y$. We write $X \sim Y$ if $X \lsm Y$ and $Y \lsm X$.
By writing $f(x) = O(g(x))$, we mean $|f(x)| \lsm g(x)$. We say that
$f$ has Fourier support in a set $\Om$ if its Fourier transform $\wh{f}$ is supported
in $\Om$.

To prepare the reader for the myriad of intervals that will occur later in 
Sections \ref{mainlemma} and \ref{mainproof}, there will be three
types of interval lengths: intervals named with a ``$K$" will be associated to
the smallest scale $\delta$, intervals named with a ``$J$" will be associated to the
intermediate scale $\nu \approx \delta^{1/k}$, and intervals named with an ``$I$" will be associated to
the largest scale $\kappa \approx \delta^{\vep}$ (though on a first reading, it might be easier to set $\kappa = 1/q$).
Finally, in the context of the decoupling constant $\mathfrak{D}_{p}(\delta)$, defined in \eqref{def:dec_const} below, we call
$p$ subcritical if $p < k(k + 1)$ and $p$ supercritical if $p \geq k(k + 1)$ (rather than the more accurate but slightly more clumsy ``not subcritical").

\subsection{Analysis over $\F$ and decoupling}
Fix a degree $k \geq 2$ and a prime number $q$ with $q > k$. We reserve the letter $p$ for the Lebesgue exponent in the main Theorem~\ref{thm:main}.
We very briefly review the harmonic analysis over $\F$ needed to set up the statement of decoupling. 
See also Section \ref{extraQp} and \cite[Section 2]{GLY21} for further discussion surrounding
the harmonic analysis and basic geometric facts over $\F$ that are useful in decoupling. Additionally
see Chapters 1 and 2 of \cite{Taibleson} and Chapter 1 (in particular Sections 1 and 4) of \cite{VVZ} for a more complete discussion of analysis on $\F$.

The field $\F$ is the completion of $\Q$ under the $q$-adic norm, defined by $|0| = 0$ and $|q^a b/c| = q^{-a}$ if $a \in \Z$, $b, c \in \Z \setminus \{0\}$ and $q$ is relatively prime to both $b$ and $c$. Then $\F$ can be identified (bijectively) with the set of all formal series
\begin{align*}
\F = \Big\{\sum_{j=k}^{\infty} a_j q^j : k \in \Z, a_j \in \{0, 1, \dots, q-1\} \text{ for every $j \geq k$} \Big\},
\end{align*}
and the $q$-adic norm on $\F$ satisfies
$|\sum_{j=k}^{\infty} a_j q^j| = q^{-k}$ if $a_k \ne 0$.
Strictly speaking we should be writing $|\cdot|_{q}$ instead of $|\cdot|$, but we omit this dependence as $q$ is fixed.
The $q$-adic norm on $\Q_q$ induces a norm on $\Q_q^k$, which we denote also by $|\cdot|$ by abuse of notation, via
$|(\xi_1,\dots,\xi_k)| := \max_{1 \leq i \leq k} |\xi_i|.$
Of particular importance is the ultrametric inequality:
$|\xi+\eta| \leq \max\{|\xi|,|\eta|\}$
with equality if $|\xi| \ne |\eta|$.
An interval in $\Q_q$ is then a set of the form $\{\xi \in \Q_q \colon |\xi - a| \leq r\}$, where $a \in \Q_q$ and $r \geq 0$; $r$ will then be called the length of the interval. 
We also will use $|I|$ to denote the length of an interval $I$.
The ring of integers $\O$ coincides with the unit interval $\{\xi \in \Q_q \colon |\xi| \leq 1\}$.
A cube in $\Q_q^k$ of side length $r$ is then a product of $k$ intervals in $\Q_q$ of lengths $r$.
% The ring of integers $\O$ is then the unit ball in $\Q_q$ given by $\{\xi \in \Q_q \colon |\xi| \leq 1\}$.
% A set of the form $\{\xi \in \Q_q^k \colon |\xi - a| \leq r\}$, where $a \in \Q_q^k$ and $r \geq 0$, will be called a cube centered at $a$ of side length $r$.
We will work with Schwartz functions defined on $\Q_q^k$ (i.e. finite linear combinations of characteristic functions of cubes in $\Q_q^k$). 
The Fourier transform of such a function $f$ will be given by 
\[
\widehat{f}(\xi) := \int_{\Q_q^k} f(x) \chi(-x \cdot \xi) dx
\]
where $\chi$ is a fixed element in the Pontryagin dual $\widehat{\Q_q}$ of $\Q_q$ that restricts to the principal character on the additive subgroup $\O$ and restricts to a non-principal character on the additive subgroup $q^{-1} \O$, $x \cdot \xi = \sum_{i=1}^k x_i \xi_i$ if $x = (x_1,\dots,x_k)$ and $\xi = (\xi_1,\dots,\xi_k)$, and $dx$ is the Haar measure on the additive group $\Q_q^k$ normalized so that $\int_{\O^k} dx = 1$.
One key property of the Fourier transform that we will use is that $\wh{1_{\O}} = 1_{\O}$, that is, the Fourier transform of the unit ball is the unit ball, see \cite[p.42]{VVZ} for a proof.

We are interested in the unit moment curve
\[
\gamma(t) := (t,t^2,\dots,t^k), \quad |t| \leq 1.
\]
For $\delta \in q^{-\N}$ and any interval $I \subset \Q_q$ with length $\geq \delta$, let $P_{\delta}(I)$ be a partition of $I$ into intervals of length $\delta$. Write $P_{\delta}$ for $P_{\delta}(\O)$.
To each interval $I \subset \O$, one associates a parallelepiped
\[
\theta_I := \Big\{\gamma(a)+\sum_{j=1}^k t_j \gamma^{(j)}(a) \in \Q_q^k \colon |t_j| \leq |I|^j \text{ for all $1 \leq j \leq k$} \Big\}
\]
of dimensions $|I| \times |I|^2 \times \cdots \times |I|^k$ where $a \in I$; this parallelepiped is independent of the choice of $a \in I$. Note that $\bigcup_{K \in P_{\delta}} \theta_K$ is a covering of a $\delta^k$ neighborhood of the unit moment curve (in fact it covers a suitable anisotropic neighborhood of that curve).
One also associates to each $K \in P_{\delta}$ a cube
\begin{align}\label{taukdef}
\tau_K := \{(\xi_1,\dots,\xi_k) \in \Q_q^k \colon |\xi_j-a^j| \leq \delta \text{ for all $1 \leq j \leq k$}\}
\end{align}
of side length $\delta$, where $a \in K$; again this is independent of the choice of $a \in K$. 
Note that for each $K \subset P_{\delta}$,
the ultrametric inequality gives that $\ta_K \subset \tau_K$.

For an interval $I \subset \O$, let $f_{I}$ be defined such that $\wh{f_I} := \wh{f} \cdot 1_{I \times \F^{k-1}}$.
For $p \geq 2$ and $\delta \in q^{-\N}$, let $\mathfrak{D}_p(\delta)$ be the smallest constant such that the inequality
\begin{equation}\label{def:dec_const}
\|f\|_{L^{p}(\F^k)} \leq \mathfrak{D}_p(\delta) (\sum_{K \in P_{\delta}} \|f_K\|_{L^{p}(\F^k)}^2)^{1/2}
\end{equation}
holds for every Schwartz function $f$ on $\Q_q^k$ with its Fourier transform $\wh{f}$ supported on $\bigcup_{K \in P_{\delta}} \theta_K$. Note that $f = \sum_{K \in P_{\delta}} f_K$.
Bourgain, Demeter, and Guth \cite{BDG} showed that
\begin{align}\label{decbdg}
\mathfrak{D}_{p}(\delta) \lsm_{\vep, p, q}\delta^{-\vep}(1 + \delta^{-(\frac{1}{2} - \frac{k(k + 1)}{2p})}),
\end{align}
and this estimate is sharp.
Strictly speaking \cite{BDG} proves a decoupling theorem over $\R$ rather
than over $\F$, but the same proof can be used to derive \eqref{decbdg}.
Choosing $f$ to be a sum of Dirac deltas
immediately implies \eqref{eq:VMT}.

\subsection{The main result}
By interpreting the refinement of Karatsuba's argument for \eqref{eq:VMTold} into decoupling language, our main result is then the following Fourier decoupling analogue of \eqref{eq:VMTold}.
In the same way \eqref{eq:VMTold} is a weaker partial result towards \eqref{eq:VMT}, Theorem \ref{thm:main} and 
Corollary \ref{maincor} should be viewed as the analogous weaker counterpart
of the sharp bound \eqref{decbdg}.

\begin{thm} \label{thm:main}
Let $p_0 \in 2\N$ be an even integer and let $c(p_0) \geq 0$ be such that
\begin{align}\label{p0assumption}
\mathfrak{D}_{p_0}(\delta) \leq C_1 \delta^{-(\frac{1}{2}-\frac{k(k+1)}{2p_0})-\frac{c(p_0) }{p_0}(1-\frac{1}{k})^{p_{0}/(2k)}}  \quad \text{for all $\delta \in q^{-\N}$}
\end{align}
where $C_1$ is independent of $\delta$. 
If $p \in p_0 + 2k\N$ and $0 < \varepsilon < 1$, then
\begin{equation} \label{eq:decbdd}
\mathfrak{D}_p(\delta) \lsm_{p, \vep, C_1}q^{a(p, p_0)/p}\delta^{-(\frac{1}{2}-\frac{k(k+1)}{2p})-\frac{c(p_0) }{p}(1-\frac{1}{k})^{p/(2k)} - \varepsilon} \quad \text{for all $\delta \in q^{-\N}$}
\end{equation}
where
\begin{align}\label{adef}
a(p, p_0) := (\frac{p - p_0}{2k})(\frac{p_0}{2} +\frac{k^2 + 7k - 4}{2}) + \frac{k}{2}(\frac{p - p_0}{2k})(\frac{p - p_0}{2k} + 1).
\end{align}
\end{thm}
Since $\mf{D}_{p}(\delta) \geq 1$ for all $p$, \eqref{p0assumption}
implies that $c(p_0)$, $k$, and $p_0$ are such that
\begin{align}\label{positiveexp}
\frac{1}{2} - \frac{k(k + 1)}{2 p_0} + \frac{c(p_0)}{p_0}(1 - \frac{1}{k})^{\frac{p_0}{2k}} \geq 0.
\end{align}
It is also known that $\mathfrak{D}_{2k}(\delta) \lsm_{\vep} \delta^{-\vep}$ for any $\vep > 0$, see for example \cite[Exercise 11.19]{Demeter-book} for the Euclidean case; we provide a proof for the case over $\F$ in the appendix for the convenience of the reader. We also remark that \cite{HW} proved, in the case of local fields, a related square function estimate with a bound independent of $\delta$ if the $f_K$'s are Fourier supported in a $\delta^{k}$ neighborhood of $\gamma(K)$; see also \cite{GGPRY} and \cite{BBH} for similar estimates.
Choosing $p_0 = 2k$ and $c(p_0) = k^{2}/2 + \vep$ for any $\vep > 0$ in applying Theorem~\ref{thm:main} we obtain:
\begin{cor}\label{maincor}
Let $p \in 2k\N$ and $0 < \vep < 1$. Then
\begin{align*}
\mathfrak{D}_{p}(\delta) \lsm_{p, \vep}q^{O(k + p/k)}\delta^{-(\frac{1}{2} - \frac{k(k + 1)}{2p}) - \frac{k^2}{2p}(1 - \frac{1}{k})^{p/(2k)} - \vep} \quad \text{for all $\delta \in q^{-\N}$}
\end{align*}
where the implied constant in the exponent of $q$ is absolute (and independent of $k$).
\end{cor}
The exponent of $q$ in Corollary \ref{maincor} is more precisely $\frac{a(p, 2k)}{p} = (\frac{1}{2k} - \frac{1}{p})\frac{k^2 + 9k - 4}{2} + \frac{1}{4}(\frac{p}{2k} - 1)$, but we opt to write it as above since it more clearly illustrates what the main terms are.
Note that the hypothesis in Theorem~\ref{thm:main} is always satisfied if $p_0$ is any fixed exponent $\geq 2$ and $c(p_0)$ is chosen large enough.
One can view Theorem~\ref{thm:main} as a way of upgrading trivial $l^{2}L^{p_0}$ decoupling at say some subcritical $p$ to $l^{2}L^p$ decoupling for all large $p$ with only a loss that decreases exponentially as $p \to +\infty$. 
Of course, if one already knew the sharp estimate in the critical $p_0 = k(k + 1)$ case,
then Theorem \ref{thm:main} implies that we know the sharp decoupling estimate for all $p \in k(k + 1) + 2k\N$.
However this already follows from interpolating the critical estimate with the trivial $l^{2}L^{\infty}$
decoupling estimate.

Though Corollary \ref{maincor} implies \eqref{eq:VMTold} with an extra $X^{\vep}$ that comes from the $\delta^{-\vep}$ factor in Corollary \ref{maincor}, 
Corollary \ref{maincor} is more general and this 
extra $\delta^{-\vep}$ term comes from needing some additional uniformity 
in the case of the general $f$ Fourier supported in $\bigcup_{K \in P_{\delta}}\ta_K$ and
an application of the broad-narrow argument to get around the use of the Prime Number Theorem in the proof of \eqref{eq:VMTold} (see Section \ref{broadnarrow}).
See Sections \ref{interpret14} and \ref{pigeon} for some more discussion comparing
the VMVT case and the general $f$ decoupling case.

We end with some discussion about how the proof of Corollary \ref{maincor} (and Theorem \ref{thm:main}) contrasts with modern decoupling
proofs of degree $k$ moment curve decoupling \cite{BDG, GLYZ} which prove \eqref{decbdg}.
Unlike the argument in \cite{BDG, GLYZ}, we are missing
any lower dimensional decoupling input and while we do use induction on scales, the iteration itself is unique in that it iterates on the $p$
in $l^{2}L^{p}$ decoupling.
Schematically, the iteration to prove Theorem \ref{thm:main} controls $l^{2}L^{p}$ decoupling by $l^{2}L^{p - 2k}$ decoupling at a larger scale.
After $O(p/k)$ steps, we are reduced to $l^{2}L^{2k}$ decoupling for the degree $k$ moment curve which follows (essentially) from
the Newton-Girard identities.
The iteration is surprisingly efficient when it controls $l^{2}L^{p}$ decoupling by $l^{2}L^{p - 2k}$ decoupling as long
as both $p$ and $p - 2k$ are supercritical. However after about $\frac{1}{2k}(p - \frac{k(k + 1)}{2})$ steps, we enter the subcritical
regime for which the iteration becomes inefficient and this is why we accrue an additional $\delta^{-\frac{k^2}{2p}(1 - \frac{1}{k})^{p/(2k)}}$ term.
When $k = 2$, the argument for Corollary \ref{maincor} uses $O(p)$ steps to prove a weak non-sharp $l^{2}L^{p}$ decoupling estimate. 
This is to be compared to the modern proof of decoupling for the parabola where to prove the sharp critical $l^{2}L^{6}$ decoupling,
one uses $O(\vep^{-1})$ many steps (see for example the proof of \cite[Lemma 2.12]{Li18}).
In the harmonic analysis literature, iterating on $p$ is not a new idea as such an argument was already used by Drury \cite{Drury} to prove cubic moment curve restriction,
though we believe this is the first time such an argument has appeared in the decoupling literature.
See also \cite{AM} by the fourth author for a similar idea in the additive combinatorics literature which was recently used to obtain
diameter free estimates for the quadratic VMVT.

Additionally, at each iterative step, three scales are key: the smallest scale $\delta$, the intermediate scale $\delta^{1/k}$, and the largest
scale $1$ (though strictly speaking in our proof the largest scale is actually $\delta^{\vep}$ rather than 1 for technical reasons). This can be compared
to \cite{BDG, GLYZ} which uses scales $\delta$, $\delta^{\vep}$ and $1$.

This paper is organized as follows: In Section \ref{extraQp}, we review some basic geometric and harmonic analysis facts in $\F$ that will be used throughout this paper. In Section \ref{sec:karatsuba}, we review the refinement of the 1973 argument of Karatsuba at a high level. In Section \ref{mainlemma}, we prove Lemma \ref{lem:main} which is the main lemma that is used to prove Theorem \ref{thm:main}. This is accomplished via combining a standard broad-narrow argument in Section \ref{broadnarrow} and some geometric properties of the moment curve that use the Newton-Girard identites, see Lemma \ref{lem:counting}. In Section \ref{mainproof}, we dyadically pigeonhole to obtain some uniformity in our estimates and prove Theorem \ref{thm:main} and Corollary \ref{maincor}. Finally, in the appendix, we include a proof of $\mathfrak{D}_{2k}(\delta) \lsm_{\vep} \delta^{-\vep}$ for completeness.

\subsection*{Acknowledgements}
This question was first posed to the third and sixth author by Shaoming Guo when the third author was visiting the Department of Mathematics at the Chinese University of Hong Kong in July 2019. This question was posed again by Shaoming Guo during a problem session at the Arithmetic (and) Harmonic Analysis workshop held (virtually) at the Mittag-Leffler Institute in early June 2021 and this current collaboration arose from that particular workshop.

KH is supported by the Additional Funding Programme for Mathematical Sciences, delivered by EPSRC (EP/V521917/1) and the Heilbronn Institute for Mathematical Research, 
ZL is supported by NSF grant DMS-1902763, 
AM is supported by Ben Green’s Simons Investigator Grant, ID 376201, 
OR is supported by the joint FWF-ANR project Arithrand: FWF: I 4945-N and ANR-20-CE91-0006, 
and P-L.Y is supported by a Future Fellowship FT20010039 from the Australian Research Council.
ZL would also like to thank the National Center for Theoretical Sciences (NCTS) in Taipei, Taiwan for their kind hospitality during his visit, where part of this work was written. We also acknowledge kind support from the American Institute of Mathematics through the Fourier restriction research community.

\section{Wavepacket decomposition and some basic geometric facts}\label{extraQp}
Throughout this paper, we will make use of wavepacket decomposition which
allows us to decompose a function $f$, which is Fourier supported in some
$\ta_K$, into linear combinations of indicator functions of translates of the parallelpiped ``dual" to $\ta_K$.
That the $q$-adic character $\chi$ is trivial on $\O$ gives
a much cleaner wavepacket decomposition when working over $\F$ than over $\R$.
See \cite[Section 3]{Tao-Recent} or \cite[Section 2.4]{Guth-325} for some discussion about wavepacket
decomposition over $\R$ in the context of the paraboloid
(though the same ideas apply for the degree $k$ moment curve).

%The following geometric considerations will be useful in formulating and establishing the uncertainty principle (also sometimes known as the ``local constancy principle") for Schwartz functions whose Fourier transform is supported in $\theta_K$.

Fix $\delta \in q^{-\N}$. It will be convenient to introduce the shorthand
\[
\theta_{\delta} := \delta \O \times \delta^2 \O \times \dots \times \delta^k \O
\]
and 
\[
T_{\delta} := \delta^{-1} \O \times \delta^{-2} \O \times \dots \times \delta^{-k} \O.
\]
They are dual to each other in the sense that
\[
T_{\delta} = \{x \in \Q_q^k \colon |x \cdot \xi| \leq 1 \text{ for all $\xi \in \theta_{\delta}$}\}.
\]
Since for any $1 \leq j \leq k$, any interval in $\Q_q$ of length $\delta^{j}$ is the disjoint union of $\delta^{-(k-j)}$ many intervals of length $\delta^{k}$, it follows that $\theta_{\delta}$ is the disjoint union of $\delta^{-\frac{k(k-1)}{2}}$ many cubes of side lengths $\delta^k$ in $\Q_q^k$. Similarly, any cube in $\Q_q^k$ of side length $\delta^{-k}$ is a disjoint union of $\delta^{-\frac{k(k-1)}{2}}$ many translates of $T_{\delta}$.

Now for $a \in \O$, let $M_{a}$ be the $k \times k$ lower-triangular matrix given by
\[
M_a = (\gamma'(a)\,\, \gamma''(a)\,\, \cdots\,\, \gamma^{(k)}(a))
\]
where we view $\gamma^{(j)}(a)$ as a column vector. Then for any $K \in P_{\delta}$, we have
\begin{equation} \label{eq:thetaK_Ma}
\theta_K = \gamma(a) + M_a \theta_{\delta}
\end{equation}
for any $a \in K$.
In fact, the right hand side is independent of $a \in K$ since if $b \in K$, then 
\[
\gamma(b) = \gamma(a) + \sum_{j=1}^k (j!)^{-1} \gamma^{(j)}(a)(b-a)^j \in \gamma(a) + M_a \theta_{\delta},
\]
and
\begin{equation} \label{eq:Ma_to_Mb}
M_a = M_b 
\left(
\begin{array}{cccc}
1 & 0 & \dots & 0 \\
(1!)^{-1}(b-a) & 1 & \dots & 0 \\
(2!)^{-1}(b-a)^2 & (1!)^{-1}(b-a) & \dots & 0 \\
\vdots & & \ddots &\\
((k-1)!)^{-1}(b-a)^{k-1} & ((k-2)!)^{-1}(b-a)^{k-2} & \dots & 1
\end{array}
\right)
\end{equation}
where the second matrix on the right hand side preserves $\theta_{\delta} = \delta \O \times \delta^2 \O \times \dots \times \delta^k \O$ (here we have used the fact that $|k!|=1$ in $\Q_q$ since $q > k$). 

For $K \in P_{\delta}$ and any $a \in K$, let $T_{0, K}$ be the dual parallelepiped to $\theta_K$ centered at the origin given by
\begin{align*}
    T_{0, K} = \{x \in \Q_q^k \colon |x \cdot (\xi-\gamma(a))|&\leq 1 \text{ for all $\xi \in \theta_K$}\}. 
\end{align*}
Using \eqref{eq:thetaK_Ma}, it is not hard to see that
\begin{align*}
T_{0, K} &= \{x \in \Q_q^k \colon |x \cdot \gamma^{(j)}(a)| \leq \delta^{-j} \text{ for all $1 \leq j \leq k$}\} \\
&= \{x \in \Q_q^k \colon M_a^{T} x \in T_{\delta} \} = M_{a}^{-T} T_{\delta}
\end{align*}
for any $a \in K$.
This parallelepiped depends only on $K$ but not on the choice of $a \in K$, since \eqref{eq:Ma_to_Mb} shows that
\[
M_a^{-T} = M_b^{-T} \left(
\begin{array}{ccccc}
1 & O(\delta) & O(\delta^2) & \dots & O(\delta^{k-1}) \\
0 & 1 & O(\delta) & \dots & O(\delta^{k-2}) \\
0 & 0 & 1 & \dots & O(\delta^{k-3}) \\
\vdots & & & \ddots &\\
0 & 0 & 0 & \dots & 1
\end{array}
\right),
\]
where $O(\delta^j)$ is some number in $\Q_q$ with norm $\leq \delta^j$, and the second matrix on the right hand side is a bijection that preserves $T_{\delta}$ by the ultrametric inequality.

\begin{lemma} \label{lem:geometry}
Let $\delta \in q^{-\N}$ and fix $K \in P_{\delta}$. 
Then 
\begin{enumerate}[(i)] 
\item \label{lem:geometryi} $\theta_K-\theta_K$ is the disjoint union of $\delta^{-\frac{k(k-1)}{2}}$ cubes of side lengths $\delta^{k}$, and
\item \label{lem:geometryii} every cube of side length $\delta^{-k}$ in $\Q_q^k$ is the disjoint union of $\delta^{-\frac{k(k-1)}{2}}$ many translates of $T_{0, K}$.
\end{enumerate}
\end{lemma}

\begin{proof}
%The proof uses the easily verified fact that for any $a \in \O$, the map $M_a^{-T}$ is a bijection on $\Q_q^k$ that maps any cube of side length $\geq 1$ to another cube of the same side length. 
\begin{enumerate}[(i)]
    \item Recall that $\theta_{\delta}$ is the disjoint union of $\delta^{-\frac{k(k-1)}{2}}$ cubes of side lengths $\delta^k$. Since $M_a$ is a bijection that maps cubes of side length $\delta^k$ to cubes of side length $\delta^k$ for any $a \in K$, and $\theta_K-\theta_K = M_a \theta_{\delta}$ for any $a \in K$, the assertion follows.
    Note that $\ta_K - \ta_K$ is just a translation of $\ta_K$ to the origin.
    \item Recall that any cube in $\Q_q^k$ of side length $\delta^{-k}$ is a disjoint union of $\delta^{-\frac{k(k-1)}{2}}$ many translates of $T_{\delta}$. Since $M_a^{-T}$ is a bijection that maps cubes of side length $\delta^{-k}$ to cubes of side length $\delta^{-k}$ for any $a \in K$, and $T_{0, K} = M_a^{-T} T_{\delta}$ for any $a \in K$, the assertion follows.
\end{enumerate}
\end{proof}

From Lemma~\ref{lem:geometry}\eqref{lem:geometryii}, we may deduce that translates of $T_{0, K}$ tile $\Q_q^k$; we denote the collection of such translates by $\mathbb{T}(K)$.
We are now ready to state the version of wavepacket decomposition that we will use.
\begin{lemma}[Wavepacket decomposition]\label{wavepacket}
Let $\delta \in q^{-\N}$ and fix $K \in P_{\delta}$.  	
Let $g$ be a Schwartz function with Fourier transform supported in $\ta_K$.
Then $|g|$ is constant on every $T \in \T(K)$, and $\widehat{g 1_T}$ is supported on $\theta_K$ for every $T \in \T(K)$. Hence it is natural to write
\begin{align}\label{wavedecomp}
g = \sum_{T \in \T(K)} g 1_T,
\end{align}
where each term $g1_{T}$ (which we will call a ``wavepacket") is Fourier supported on $\theta_K$ and has constant modulus on every $T \in \T(K)$. 
It also follows that if $\mathcal{T}$ is any subset of $\T(K)$, then $\sum_{T \in \mathcal{T}} g 1_T$ is Fourier supported in $\theta_K$.
\end{lemma}

\begin{proof}
First, to prove that $|g|$ is constant on any translates of $T_{0, K}$, one only needs to prove the case when $\delta = 1$, $K = \Z_{q}$, and then apply a change of variables, but we opt for a more explicit proof.
We will show that $|g(x)|$ is constant for all $x \in A + T_{0, K}$ for any
$A \in \Q_{q}^{k}$. By Fourier inversion we have that
\begin{align*}
|g(x)| &= |\int_{\ta_K}\wh{g}(\xi)\chi(\xi \cdot x)\, d\xi|\\
&= |\int_{|t_1| \leq \delta, \ldots, |t_k| \leq \delta^{k}}\wh{g}(\gamma(a) + \sum_{j =1 }^{k}t_{j}\gamma^{(j)}(a) )\chi([\gamma(a) + \sum_{j =1 }^{k}t_{j}\gamma^{(j)}(a)] \cdot x)\, dt|\\
&= |\int_{|t_1| \leq \delta, \ldots, |t_k| \leq \delta^k}\wh{g}(\gamma(a) + M_{a}t)\chi(M_{a}^{T}x \cdot t)\, dt|.
\end{align*}
For $x \in A + T_{0, K}$, we write $x = A + M_{a}^{-T}y'$ where $|y'_{j}| \leq \delta^{-j}$ for $j = 1, 2, \ldots, k$.
Therefore
\begin{align*}
|g(x)| &= |\int_{|t_1| \leq \delta, \ldots, |t_k| \leq \delta^k}\wh{g}(\gamma(a) + M_{a}t)\chi(M_{a}^{T}A\cdot t)\chi(y' \cdot t)\, dt|\\
&= |\int_{|t_1| \leq \delta, \ldots, |t_k| \leq \delta^k}\wh{g}(\gamma(a) + M_{a}t)\chi(M_{a}^{T}A\cdot t)\, dt|
\end{align*}
where we have used that $y' \cdot t \in \Z_{q}$, and so $\chi(y' \cdot t) = 1$.
The right hand side is then independent of $y'$ and so the 
above equality is true for all $x \in A + T_{0, K}$. In particular this shows that
$|g|$ is a constant on $A + T_{0, K}$. This constant depends on $K$, $g$ and $A$, but is a constant nonetheless. 

Next, to prove that $\widehat{g 1_T}$ is supported on $\theta_K$, it suffices to observe that $\widehat{g 1_T} = \widehat{g}*\widehat{1_T}$, and that $\widehat{1_T}$ is supported on $\theta_K - \theta_K$ for every $T \in \T(K)$: in fact, for every $T \in \T(K)$, $\widehat{1_T}$ is a modulation of $\widehat{1_{T_{0,K}}}$, and if $a$ is any point in $K$, then $T_{0,K} = M_a^{-T} T_{\delta}$. It follows that
\[
\begin{split}
\widehat{1_{T_{0,K}}}(\xi) &= \int_{M_a^{-T}T_{\delta}} \chi(-x \cdot \xi) dx \\
&= \det(M_a)^{-1} \int_{T_{\delta}} \chi(-M_a^{-T} y \cdot \xi) dy = \det(M_a)^{-1} \delta^{-k(k+1)/2} 1_{\theta_{\delta}}(M_a^{-1} \xi)
\end{split}
\]
is supported on $M_a \theta_{\delta} = \theta_K - \theta_K$. 
Finally, the decomposition \eqref{wavedecomp} follows since parallelepipeds in $\mathbb{T}(K)$ tile $\Q_q^k$.
This completes the proof of the lemma.
\end{proof}

\section{Sketch of the Karatsuba argument}\label{sec:karatsuba}

Before we dive into the proof of Theorem \ref{thm:main}, we review the proof of \eqref{eq:VMTold} with an eye towards
interpreting each step into decoupling language. See also, for example,
\cite[Section 5.1]{V} or \cite[Theorem 13 - Lemma 21]{Tao254A} for more details of the number theoretic argument. 
Just for this section, we revert back to calling $p$ a prime so as to best match these references.

\subsection{Step 1: Introducing some $p$-adic separation} \label{sect:3.1} Given $X \geq 1$, one finds, using the Prime Number Theorem, a prime $p \sim X^{1/k}$ such that
$J_{s, k}(X)$ is controlled by $J_{s, k}(X, p)$, where $J_{s, k}(X, p)$ is defined to be the number of solutions $(x_1, \ldots, x_s, y_1, \ldots, y_s) \in ([1, X] \cap \N)^{2s}$
to \eqref{eq:Vinosystem} with the additional condition that $x_{1}, \ldots, x_k$ are pairwise distinct mod $p$ and $y_{1}, \ldots, y_k$ are
pairwise distinct mod $p$. Since $p$ is rather large, this is a rather mild condition and so we heuristically should still
expect $J_{s, k}(X) \approx J_{s, k}(X, p)$. The benefit of this extra $p$-adic separation (transversality) 
in these $2k$ variables is that we will get to apply Linnik's Lemma (in Step 3, \eqref{linnik} below)
which will up to permutation uniquely determine these variables.

\subsection{Step 2: Applying the union bound/H\"{o}lder}  \label{sect:3.2} We now write $J_{s, k}(X, p)$ as
\begin{align*}
\int_{[0, 1]^{2s}}|\sum_{\st{a_1, \ldots, a_k \Mod{p}\\\text{$a_{i}$ pairwise distinct}}}\prod_{j = 1}^{k}\sum_{\st{n_j \equiv a_{j}\Mod{p}\\1 \leq n_j \leq X}}e(n_j\alpha_1+ \cdots + n_{j}^{k}\alpha_k)|^{2}|\sum_{1 \leq n \leq X}e(n\alpha_1 + \cdots + n^{k}\alpha_k)|^{2s - 2k}\, d\alpha.
\end{align*}
Write $|\sum_{1 \leq n \leq X}|^{2s - 2k} = |\sum_{a\Mod{p}}\sum_{n \equiv a \Mod{p}}|^{2s - 2k}$ and apply H\"{o}lder's inequality to control the above by
\begin{align}
\begin{aligned}\label{jskxpa}
p^{2s - 2k}\max_{a \Mod{p}}\int_{[0, 1]^{2s}}|\sum_{\st{a_1, \ldots, a_k \Mod{p}\\\text{$a_{i}$ pairwise distinct}}}\prod_{j = 1}^{k}&\sum_{\st{n_j \equiv a_{j}\Mod{p}\\1 \leq n_j \leq X}}e(n_j\alpha_1+ \cdots + n_{j}^{k}\alpha_k)|^{2}\times\\
&\quad\quad|\sum_{\st{n \equiv a\Mod{p}\\1 \leq n \leq X}}e(n\alpha_1 + \cdots + n^{k}\alpha_k)|^{2s - 2k}\, d\alpha.
\end{aligned}
\end{align}
Denote the integral above to be $J_{s, k}(X, p, a)$. This expression counts the number of solutions 
$(x_1, \ldots, x_s, y_1, \ldots, y_s) \in ([1, X] \cap \N)^{2s}$ to \eqref{eq:Vinosystem} with $x_{1}, \ldots, x_k$ pairwise distinct mod $p$, $y_{1}, \ldots, y_k$ pairwise distinct mod $p$, and $x_{k + 1} \equiv \cdots \equiv x_s \equiv y_{k + 1} \equiv \cdots \equiv y_s \equiv a \Mod{p}$.

\subsection{Step 3: Solution counting}  \label{sect:3.3} Translation invariance of the Vinogradov system implies that we may bound $J_{s, k}(X, p, a)$ by $J_{s, k}(X, p, 0)$.
Rearrange the Vinogradov system \eqref{eq:Vinosystem} as
\begin{align}\label{step3eq1}
x_{k + 1}^{j} + \cdots + x_{s}^{j} - y_{k + 1}^{j} - \cdots - y_{s}^{j} &=  y_{1}^{j} + \cdots + y_{k}^{j} - x_{1}^{j} - \cdots - x_{k}^{j}, \quad 1 \leq j \leq k
\end{align}
where $x_1, \ldots, x_k$ are distinct mod $p$ and $y_1, \ldots, y_k$ are distinct mod $p$ and 
since we are considering $J_{s, k}(X, p, 0)$, we have that $x_{k + 1}, \ldots, x_s, y_{k + 1}, \ldots, y_s \equiv 0 \Mod{p}$.
Each choice of $x_1, \ldots, x_k, y_1, \ldots, y_k$ gives $\leq J_{s - k, k}(X/p)$ many solutions to
$(x_{k + 1}, \ldots, x_s, y_{k + 1}, \ldots, y_s)$. To see this, write the count
for \eqref{step3eq1} as an integral and use the triangle inequality; the basic idea being that shifts of the Vinogradov system can only give fewer solutions. 

Next, fixing one of the at most $J_{s - k, k}(X/p)$ many tuples $(x_{k + 1}, \ldots, x_s, y_{k + 1}, \ldots, y_s)$, how many valid $x_{1}, \ldots, x_k, y_1, \ldots, y_k$ are there? 
Since requiring $y_1, \ldots, y_k$ to be distinct mod $p$ is a rather mild condition, there are $\leq X^{k}$ such $(y_1, \ldots, y_k)$.
Any valid $(x_1, \ldots, x_k) \in ([1, X] \cap \N)^{k}$ must satisfy
\begin{align*}
x_{1}^{j} + \cdots + x_{k}^{j} \equiv H_{j} \Mod{p^j}, \quad 1 \leq j \leq k
\end{align*}
where the $x_{i}$ are pairwise disjoint mod $p$ for some $H_j$ that depends on $(y_1, \ldots, y_k)$ (of which there are $\leq X^k$ many possibilities) and 
$(x_{k + 1}, \ldots, x_s, y_{k + 1}, \ldots, y_s)$ (of which there are $\leq J_{s - k, k}(X/p)$ many possibilities).
Since $p^{k} \sim X$, instead of counting integers between $1$ and $X$, we can count the $x_i$ mod $p^k$. 
Thus it remains to count the number of residue classes $(x_1 \Mod{p^k}, \ldots, x_k \Mod{p^k})$ such that
\begin{align}\label{linnik}
x_{1}^{j} + \cdots + x_{k}^{j} \equiv H_{j} \Mod{p^j}, \quad 1 \leq j \leq k
\end{align}
and $x_{i} \Mod{p^k}$ are pairwise distinct mod $p$.
Linnik's Lemma \cite{L43} then says that there are at most $k!p^{k(k - 1)/2}$ many such $k$-tuples of residue classes and the proof
follows from first upgrading all residue classes mod $p^j$ in \eqref{linnik} to mod $p^k$ (by paying a cost of $p^{k(k - 1)/2}$) and
then using the Newton-Girard identities which essentially uniquely determine the $x_1, \ldots, x_k$ (up to permutation).
This bound is efficient since probabilistic heuristics suggest that we should expect $\approx (p^k)^{k}/p^{k(k + 1)/2} = p^{k(k - 1)/2}$
many solutions.
Thus we have that
\begin{equation} \label{eq:conclusion3.3}
J_{s, k}(X, p, 0) \lsm_{k} J_{s - k, k}(X/p)X^{k}p^{k(k - 1)/2}.
\end{equation}

\subsection{Step 4: Iteration}
Putting Steps 1 to 3 together we obtain the iteration that
\begin{align}\label{step4}
J_{s, k}(X) \lsm_{k} p^{2s - 2k}J_{s - k, k}(X/p)X^{k}p^{k(k - 1)/2}.
\end{align}
Running this iteration about $O(s/k)$ many steps reduces to an estimate on $J_{k, k}(X)$ from which one can easily compute
there are $O(X^k)$ many solutions by the Newton-Girard identities.
The iteration \eqref{step4} is sharp if both $s$ and $s - k$ are supercritical.
If they are, then heuristically, we expect $J_{s, k}(X) \approx X^{2s - k(k + 1)/2}$ and $J_{s - k, k}(X/p) \approx (X/p)^{2(s - k) - k(k + 1)/2}$.
Then the right hand side of \eqref{step4} becomes $X^{2s}X^{-3k/2 - k^{2}/2}p^{k^2}$ which is equal to $X^{2s - k(k + 1)/2}$ since $p \sim X^{1/k}$.
However, both sides are not the same if one of $s$ or $s - k$ is subcritical. This is where the inefficiency of $X^{\frac{k^2}{2}(1 - \frac{1}{k})^{s/k}}$
comes from.

\subsection{Interpreting Steps 1-4 into decoupling}\label{interpret14} Having briefly summarized the number theoretic argument into four steps, we now briefly sketch the main points
to interpret into decoupling. 
First we discuss the scales needed in the proof. From Steps 1 and 3, there are three scales:
the largest scale $X$, the intermediate scale $p \sim X^{1/k}$, and the smallest scale 1.
Correspondingly in our proof, we use three scales: the smallest scale $\delta$,
the intermediate scale $\nu := q^{\lfloor \log_{q}\delta^{1/k}\rfloor} \sim \delta^{1/k}$, and the largest scale $1$.
For some technical reasons surrounding the broad-narrow reduction, in lieu of the scale 1, we will actually use
the scale $\kappa := q^{\lfloor \log_{q}\delta^{\vep}\rfloor}$ where $\vep$ is as in \eqref{eq:decbdd}.

Next, we discuss the reduction to the decoupling analogue of \eqref{jskxpa}.
In Step 1, two residue classes being distinct mod $p$
means they are $p$-adically separated by a distance 1 and so this should correspond to two intervals which are 1-separated. To get around the use of the Prime Number Theorem,
we make use instead of broad-narrow reduction due to Bourgain and Guth in \cite{BG} which will allow us to reduce
to controlling a multilinear decoupling expression. 

Third, the loss of $p^{2s - 2k}$ in Step 2 above deserves some mention. This loss comes from essentially having applied the union bound 
\begin{align*}
|\sum_{1 \leq n \leq X}e(n\alpha_1 + \cdots + n^{k}\alpha_k)| &= |\sum_{a \Mod{p}}\sum_{\st{n \equiv a \Mod{p}\\1 \leq n \leq X}}e(n\alpha_1 + \cdots + n^{k}\alpha_k)|\\
&\leq p\max_{a\Mod{p}}|\sum_{\st{n \equiv a \Mod{p}\\1 \leq n \leq X}}e(n\alpha_1 + \cdots + n^{k}\alpha_k)|.
\end{align*}
Heuristically we expect this inequality to be efficient since each $\sum_{n \equiv a \Mod{p}}$ contributes equally to the entire sum
as the exponential sum should not bias one residue class mod $p$ over another.
This however is not necessarily true in the decoupling case and will require us to obtain some extra uniformity via dyadic pigeonholing, see Section \ref{pigeon}, later.

Finally, to interpret the solution counting Step 3, we make use of the simple identity \[\int_{\F^k}f(x)\, dx = \wh{f}(0)\]
which converts the integral of $f$ into a question of whether $0$ is contained in the support of $\wh{f}$. This is done in Lemma \ref{lem:counting}
below and the proof relies on the Newton-Girard identities, much like in the proof of Linnik's Lemma. This part of the argument
requires that $p$ is even and is reminiscent of a C\'{o}rdoba-Fefferman argument (see for example \cite[Section 3.2]{Demeter-book} or \cite{Cordoba77, Cordoba82, Fefferman73}).

\section{The main lemma}\label{mainlemma}
One standard property about the moment
curve decoupling constant that we use is affine rescaling. This property plays the analogue of translation-dilation invariance
of the Vinogradov system \eqref{eq:Vinosystem}.
\begin{lemma}[Affine rescaling]\label{rescale}
Let $g$ be a Schwartz function on $\F^k$ Fourier supported in $\bigcup_{K \in P_{\delta}}\ta_K$. Then for any interval $I \subset \O$ of length $\kappa \geq \delta$, we have
\begin{align*}
\nms{g_I}_{L^{p}(\F^k)} \leq \mathfrak{D}_{p}(\frac{\delta}{\kappa})(\sum_{K \in P_{\delta}(I)}\nms{g_K}_{L^{p}(\F^k)}^{2})^{1/2}.
\end{align*}
\end{lemma}
\begin{proof}
This proof is standard and follows from a change of variables which can be found for example in \cite[Section 11.2]{Demeter-book}.
\end{proof}

Our main lemma in proving Theorem \ref{thm:main} is the following:
\begin{lemma}\label{lem:main}
Let $p \in 2k + 2\N$, $\delta \in q^{-\N}$ and $\kappa \in q^{-\N} \cap [\delta, 1)$. Let $\nu = q^{\lfloor \log_{q}\delta^{1/k}\rfloor} \in q^{-\N}$ so that $\nu \leq \delta^{1/k}$. If $g$ is a Schwartz function with Fourier support in $\bigcup_{K \in P_{\delta}}\ta_K$, then we have
\begin{align*}
\begin{aligned}
\int_{\F^k}&|g|^{p} \leq C\mathfrak{D}_{p}(\frac{\delta}{\kappa})^{p}(\sum_{K \in P_{\delta}}\nms{g_{K}}_{L^{p}(\F^k)}^{2})^{p/2}+ Cq^{-k(k - 1)}\kappa^{-(k^{2} + 4k - 2)}\nu^{-k(k - 1)/2} N^{p - 2k}\times\\
&\mathfrak{D}_{p - 2k}(\frac{\delta}{\nu})^{p - 2k}\max_{K \in P_{\delta}}\nms{g_K}_{L^{\infty}(\F^k)}^{k}(\sum_{\bar{K} \in P_{\delta}}\nms{g_{\bar{K}}}_{L^{\infty}(\F^k)})^{k}\max_{J \in P_{\nu}}(\sum_{K' \in P_{\delta}(J)}\nms{g_{K'}}_{L^{p - 2k}(\F^k)}^{2})^{(p - 2k)/2}
\end{aligned}
\end{align*}
where $N$ is the number of $J \in P_{\nu}$ for which $g_{J} \neq 0$ and $C$ depends only on $k$ and $p$.
\end{lemma}
Here $\kappa$ is a somewhat technical parameter that is chosen to be roughly $\delta^{\vep}$ later in Section \ref{mainproof}. However, on
a first reading, it might be more convenient for the reader to take $\kappa = 1/q$ to better grasp the moving parts of the argument.
The somewhat non-standard decoupling right hand side in Lemma
\ref{lem:main} is reminiscent of the right hand side used in Theorem 1.2 of \cite{GMW}.
To give more context to the above lemma, the following estimate is true:

\begin{lemma} \label{lem:reversed} 
For any $p > 2k$, we have
\begin{align*}
( \sum_{K \in P_{\delta}} &\|g_K\|_{L^{p}(\F^k)}^2 )^{p/2}\\
&\leq N^{(p - 2k)/2} \max_{K \in P_{\delta}} \|g_K\|_{L^{\infty}(\F^k)}^k (\sum_{\bar{K} \in P_{\delta}} \|g_{\bar{K}}\|_{L^{\infty}(\F^k)})^k \max_{J \in P_{\kappa}} ( \sum_{K \in P_{\delta}(J)} \|g_K\|_{L^{p-2k}(\F^k)}^2)^{(p - 2k)/2}
\end{align*}
where $N$ is as defined in Lemma \ref{lem:main}.
\end{lemma}

\begin{proof}
H\"{o}lder's inequality gives us
\[
\|g_K\|_{L^{p}(\F^k)} \leq \|g_K\|_{L^{\infty}(\F^k)}^{\frac{2k}{p}} \|g_K\|_{L^{p - 2k}(\F^k)}^{1-\frac{2k}{p}},
\]
and so, applying
\[
\Big(\sum_K a_K^{\frac{2k}{p}} b_K^{\frac{2k}{p}} c_K^{2(1-\frac{2k}{p})} \Big)^{\frac{p}{2}}
\leq (\max_K a_K)^k (\sum_K b_K)^k (\sum_K c_K^2)^{\frac{p-2k}{2}}
\]
with $a_K = b_K = \|g_K\|_{L^{\infty}(\F^k)}$ and $c_K = \|g_K\|_{L^{p - 2k}(\F^k)}$, we get
\[
( \sum_{K \in P_{\delta}} \|g_K\|_{L^{p}(\F^k)}^2 )^{p/2} 
\leq \max_{K \in P_{\delta}} \|g_K\|_{L^{\infty}(\F^k)}^k (\sum_{\bar{K} \in P_{\delta}} \|g_{\bar{K}}\|_{L^{\infty}(\F^k)})^k ( \sum_{K \in P_{\delta}} \|g_K\|_{L^{p-2k}(\F^k)}^2 )^{(p - 2k)/2}.
\]
It remains to observe that
\begin{align*}
( \sum_{K \in P_{\delta}} \|g_K\|_{L^{p-2k}(\F^k)}^2 )^{(p - 2k)/2} \leq N^{(p - 2k)/2}  \max_{J \in P_{\kappa}} ( \sum_{K \in P_{\delta}(J)} \|g_K\|_{L^{p-2k}(\F^k)}^2)^{(p - 2k)/2}. 
\end{align*}
\end{proof}
Suppose for a moment that in Lemma \ref{lem:reversed}, we had an equality instead of an inequality. This is indeed the case when
$g(x)$ is the exponential sum $X^{-100k^2} 1_{|x| \leq X^{100k}}\sum_{j=1}^X e(\gamma(j) \cdot x)$ that arises in using decoupling to estimate 
the number of solutions in \eqref{eq:VMTold}.
As $N \leq \nu^{-1}$ (and taking, for convenience, $\kappa = 1/q$), Lemma~\ref{lem:main} would give us
\begin{align}\label{iterationheur}
\mathfrak{D}_p(\delta)^p \leq C \, \mathfrak{D}_p(q\delta)^p + C q^{5k - 2} \nu^{-\frac{p}{2}+k-\frac{k(k-1)}{2}} \mathfrak{D}_{p-2k}(\frac{\delta}{\nu})^{p-2k} 
\end{align}
Heuristically, we expect this iteration to be efficient as long as $p - 2k$ (and so also $p$) is supercritical. 
To see this, if $r$ is supercritical,
then we heuristically expect that $\mathfrak{D}_{r}(\delta)^{r} \approx \delta^{-\frac{r}{2} + \frac{k(k + 1)}{2}}$ for all $\delta$. 
Thus the iteration should be
efficient if with this assumption on the size of $\mathfrak{D}_{r}(\delta)^{r}$, both sides of \eqref{iterationheur}
are the same. The right hand side of \eqref{iterationheur} is then 
\begin{align*}
\sim_{q} (\delta^{-\frac{1}{k}})^{\frac{p}{2}+k-\frac{k(k-1)}{2}} (\delta^{-1+\frac{1}{k}})^{\frac{p-2k}{2}-\frac{k(k+1)}{2}} = \delta^{-(\frac{p}{2} - \frac{k(k+1)}{2} )}
\end{align*}
which is comparable to the left hand side of \eqref{iterationheur}. A similar calculation shows that this iteration
is not efficient if at least one of $p$ or $p - 2k$ is subcritical.

Unfortunately the reverse inequality in Lemma~\ref{lem:reversed} fails to hold for general $g$. This is because we lack the uniformity in the exponential sum that one considers when one counts solutions to the Vinogradov system. This uniformity can be restored by pigeonholing, which only produces $\delta^{-\varepsilon}$ losses. This pigeonholing must be done before one applies induction on scales and iterates on the Lebesgue exponent $p$. The full argument is carried out in detail in Section~\ref{mainproof}.

\subsection{Proof of Lemma \ref{lem:main}}
The proof of Lemma \ref{lem:main} uses a broad/narrow dichotomy, due to Bourgain and Guth \cite{BG} combined with some basic geometric
geometric properties of the moment curve. See also for example \cite[Chapter 7]{Demeter-book}. 

\subsubsection{The broad-narrow argument}\label{broadnarrow}
First, we have the pointwise bound $|g(x)| \leq \sum_{I \in P_{\kappa}} |g_I(x)|$. At every point $x \in \F^k$, let $\mathcal{I}_x$ be the set of all intervals $I' \in P_{\kappa}$ such that $|g_{I'}(x)| \geq \kappa \max_{I \in P_{\kappa}} |g_I(x)|$. Suppose first $\mathcal{I}_x$ contains at least $k$ (disjoint) intervals, say $I'_1,\dots,I'_k$ (all dependent on $x$) of length $\kappa$ and $|g_{I'_1}(x)| = \max_{I \in P_{\kappa}} |g_I(x)|$: in this case we have
\begin{align*}
|g(x)| &\leq \kappa^{-1} \max_{I \in P_{\kappa}} |g_I(x)| \leq \kappa^{-1} \kappa^{-(k-1)/k} |g_{I'_1}(x) \dots g_{I'_k}(x)|^{1/k} \\
&\leq \kappa^{-2 + 1/k} \max_{\substack{I_1, \dots, I_k \in P_{\kappa} \\ d(I_i,I_j) > \kappa \, \forall i \ne j}} |g_{I_1}(x) \dots g_{I_k}(x)|^{1/k}.
\end{align*}
Here we used that in $\F^k$, two distinct intervals of the same length are separated by at least that length.
Alternatively, $\mathcal{I}_x$ contains at most $k-1$ intervals, in which case
\[
|g(x)| \leq \sum_{I \in \mathcal{I}_x} |g_I(x)| + \sum_{I \in P_{\kappa} \setminus \mathcal{I}_x} |g_I(x)| < (k-1) \max_{I \in P_{\kappa}} |g_I(x)| + \sum_{I \in P_{\kappa} \setminus \mathcal{I}_x} \kappa \max_{I \in P_{\kappa}} |g_I(x)| < k \max_{I \in P_{\kappa}} |g_I(x)|.
\]
As a result, we obtain the pointwise bound that for each $x \in \F^k$, we have
\[
|g(x)| \leq k \max_{I \in P_{\kappa}} |g_{I}(x)| + \kappa^{-2 + 1/k} \max_{\substack{I_1, \dots, I_k \in P_{\kappa} \\ d(I_i,I_j) > \kappa \, \forall i \ne j}} |g_{I_1}(x) \dots g_{I_k}(x)|^{1/k}
\]
which, upon raising both sides to power $2k$ and applying $(A+B)^{2k} \leq 2^{2k-1} (A^{2k}+B^{2k})$ (a consequence of the convexity of $x \mapsto x^{2k}$), yields
\begin{equation} \label{eq:g_pointwise}
|g(x)|^{2k} 
\leq 2^{2k-1} k^{2k}
\max_{I \in P_{\kappa}} |g_{I}(x)|^{2k} 
+ 2^{2k-1} \kappa^{-(4k-2)} \max_{\substack{I_1, \dots, I_k \in P_{\kappa} \\ d(I_i,I_j) > \kappa \, \forall i \ne j}} |g_{I_1}(x) \dots g_{I_k}(x)|^2.
\end{equation} 
Using this pointwise bound while integrating we find that 
\begin{align*}
\int_{\F^k}|g|^p 
&= \int_{\F^k}|g|^{2k} |g|^{p-2k} 
\\ &\leq  
C \int_{\F^k}(\max_{I \in P_{\kappa}} |g_I|^2)^k |g|^{p-2k} 
+ C \kappa^{-(4k-2)} \int_{\F^k}\max_{\substack{I_1, \dots, I_k \in P_{\kappa} \\ d(I_i,I_j) > \kappa \, \forall i \ne j}} |g_{I_1} \dots g_{I_k}|^2 |g|^{p-2k} 
\\ &\leq 
C \int_{\F^k}(\sum_{I \in P_{\kappa}} |g_I|^2)^k |g|^{p-2k} 
+ C \kappa^{-(4k-2)} \sum_{\substack{I_1, \dots, I_k \in P_{\kappa} \\ d(I_i,I_j) > \kappa \, \forall i \ne j}} \int_{\F^k}|g_{I_1} \dots g_{I_k}|^2 |g|^{p-2k}
\\ &\leq 
C \int_{\F^k}(\sum_{I \in P_{\kappa}} |g_I|^2)^k |g|^{p-2k} 
+ C \kappa^{-(4k-2)} \kappa^{-k} \max_{\substack{I_1, \dots, I_k \in P_{\kappa} \\ d(I_i,I_j) > \kappa \, \forall i \ne j}} \int_{\F^k}|g_{I_1} \dots g_{I_k}|^2 |g|^{p-2k} 
\end{align*}
for some $C$ depending on $k$.
H\"{o}lder's inequality followed by Minkowski's inequality implies that the first term satisfies
\begin{align*}
C \int_{\F^k}(\sum_{I \in P_{\kappa}} |g_I|^2)^k |g|^{p-2k} 
& \leq C \Big( \int_{\F^k}(\sum_{I \in P_{\kappa}} |g_I|^2)^{p/2} \Big)^{2k/p} \Big( \int_{\F^k}|g|^p \Big)^{(p-2k)/p} 
\\& \leq C \Big( \sum_{I \in P_{\kappa}} \|g_I\|_{L^{p}(\F^k)}^2 \Big)^{k} \Big( \int_{\F^k}|g|^p \Big)^{(p-2k)/p} 
\\& \leq \frac{1}{2} \int_{\F^k}|g|^p   + C' \Big( \sum_{I \in P_{\kappa}} \|g_I\|_{L^{p}(\F^k)}^2 \Big)^{p/2}
\end{align*}
for some $C'$ that depends on $k$ and $p$.
The last inequality uses Young's inequality and the fact that $p\geq2k$. 
Therefore, 
\[
\int_{\F^k}|g|^p 
\lesssim_{p} 
\Big( \sum_{I \in P_{\kappa}} \|g_I\|_{L^{p}(\F^k)}^2 \Big)^{p/2} 
+ \kappa^{-(5k-2)} \max_{\substack{I_1, \dots, I_k \in P_{\kappa} \\ d(I_i,I_j) > \kappa \, \forall i \ne j}} \int_{\F^k}|g_{I_1} \dots g_{I_k}|^2 |g|^{p-2k}. 
\]
Using affine rescaling (Lemma \ref{rescale}) \and applying the definition \eqref{def:dec_const} of our decoupling constant, we deduce that 
\[
\Big( \sum_{I \in P_{\kappa}} \|g_I\|_{L^{p}(\F^k)}^2 \Big)^{p/2} \leq \mathfrak{D}_p(\frac{\delta}{\kappa})^p \Big( \sum_{K \in P_{\delta}} \|g_K\|_{L^{p}(\F^k)}^2 \Big)^{p/2}.
\]
Plugging this into the above yields
\begin{equation} \label{eq:analog_step1}
\int_{\F^k}|g|^p 
\lesssim_{p}
\mathfrak{D}_p(\frac{\delta}{\kappa})^p \Big( \sum_{K \in P_{\delta}} \|g_K\|_{L^{p}(\F^k)}^2 \Big)^{p/2} 
+ \kappa^{-(5k-2)} \max_{\substack{I_1, \dots, I_k \in P_{\kappa} \\ d(I_i,I_j) > \kappa \, \forall i \ne j}} \int_{\F^k}|g_{I_1} \dots g_{I_k}|^2 |g|^{p-2k}.
\end{equation}
This inequality \eqref{eq:analog_step1} is the analogue of Step 1 in Section \ref{sect:3.1}. The requirement that we analyze solutions to the Vinogradov system with $x_1,\dots,x_s$ and $y_1,\dots,y_s$ being distinct mod~$p$ corresponds to the requirement that we analyze $\int_{\F^k} |g_{I_1} \dots g_{I_k}|^2 |g|^{p-2k}$ with $d(I_i,I_j) > \kappa$ for all $1 \leq i \ne j \leq k$ with $\kappa = 1/q$.

Next, we mimic Step 2 in Section \ref{sect:3.2}. Recalling our definition of $N$ in the statement of Lemma \ref{lem:main}, H\"{o}lder's inequality gives
\[
|g|^{p-2k} 
\leq N^{p-2k-1} \sum_{J \in P_{\nu}} |g_J|^{p-2k}.
\] 
Applying this in the second term in \eqref{eq:analog_step1}, we get
\begin{align}
&\int_{\F^k}|g|^p \notag \\
&\lesssim_{p} 
\mathfrak{D}_p(\frac{\delta}{\kappa})^p \Big( \sum_{K \in P_{\delta}} \|g_K\|_{L^{p}(\F^k)}^2 \Big)^{p/2} 
+ \kappa^{-(5k-2)} N^{p-2k-1} \max_{\substack{I_1, \dots, I_k \in P_{\kappa} \\ d(I_i,I_j) > \kappa \, \forall i \ne j}} \sum_{J \in P_{\nu}} \int_{\F^k}|g_{I_1} \dots g_{I_k}|^2 |g_J|^{p-2k} \notag \\
&\lesssim_{p} 
\mathfrak{D}_p(\frac{\delta}{\kappa})^p \Big( \sum_{K \in P_{\delta}} \|g_K\|_{L^{p}(\F^k)}^2 \Big)^{p/2} 
+ \kappa^{-(5k-2)} N^{p-2k} \max_{\substack{I_1, \dots, I_k \in P_{\kappa} \\ d(I_i,I_j) > \kappa \, \forall i \ne j}} \max_{J \in P_{\nu}} \int_{\F^k}|g_{I_1} \dots g_{I_k}|^2 |g_J|^{p-2k} \label{eq:analog_step2}
\end{align}
which is the analogue of Step 2 in Section~\ref{sect:3.2}.

To analyze the second term in \eqref{eq:analog_step2}, we fix $I_1, \dots, I_k \in P_{\kappa}$ with $d(I_i,I_j) > \kappa$ for all $i \ne j$, and fix $J \in P_{\nu}$ with $g_J \ne 0$. To estimate the integral $\int_{\F^k}|g_{I_1} \dots g_{I_k}|^2 |g_J|^{p-2k}$, first 
note that the Fourier transform of $|g_J|^2 = g_J \overline{g_J}$ is supported in the parallelepiped $\theta_J - \theta_J$, of dimension $\nu \times \nu^2 \times \dots \times \nu^k$. Since our hypothesis guarantees that $p-2k$ is an even positive integer, the same is true for the Fourier transform of $|g_J|^{p-2k}$. Lemma~\ref{lem:geometry}\eqref{lem:geometryi} applied to $J \in P_{\nu}$ instead of $K \in P_{\delta}$ shows that the Fourier support of $|g_J|^{p-2k}$ is the disjoint union of $\nu^{-k(k - 1)/2}$ many cubes of side lengths $\nu^k$, and we denote this collection of cubes by $\{\Box\}$. 
This corresponds to the fact that we have a $k$-tuple of residue
classes $(H_1 \Mod{p}, H_2 \Mod{p^2}, \ldots, H_k \Mod{p^k})$ which we can upgrade to $p^{k(k - 1)/2}$
many $k$-tuples of the form $(H_{1}' \Mod{p^k}, H_{2}' \Mod{p^k}, \ldots, H_{k}' \Mod{p^k})$. 
Note that the side length $\nu^k$ of the cubes $\Box$ is $\leq \delta$. 

We now apply Fourier inversion and turn products into convolutions. We have
\begin{align*}
\int_{\F^k}|g_{I_1} \dots g_{I_k}|^2 |g_J|^{p-2k} &= \sum_{\substack{K_i \in P_{\delta}(I_i) \\ i=1,\dots,k}} \sum_{\substack{\bar{K}_j \in P_{\delta}(I_j) \\ j=1,\dots,k}} \int_{\F^k}g_{K_1} \dots g_{K_k} \overline{g_{\bar{K}_1}} \dots \overline{g_{\bar{K}_k}} |g_J|^{p-2k} \\
&= \sum_{\Box} \sum_{\substack{K_i \in P_{\delta}(I_i) \\ i=1,\dots,k}} \sum_{\substack{\bar{K}_j \in P_{\delta}(I_j) \\ j=1,\dots,k}} \widehat{g_{K_1}} * \dots * \widehat{g_{K_k}} * \gkbbh * (\widehat{ |g_J|^{p-2k} } 1_{\Box})(0). 
\end{align*}
For each fixed $\Box$ and $\bar{K}_1 \in P_{\delta}(I_1)$, $\dots$, $\bar{K}_k \in P_{\delta}(I_k)$, let $S(\bar{K}_1,\dots,\bar{K}_k,\Box)$ be the set of all $(K_1,\dots,K_k)$ with $K_i \in P_{\delta}(I_i)$ such that 
\begin{align}\label{momentgeom}
0 \in \supp (\widehat{g_{K_1}} * \dots * \widehat{g_{K_k}} * \gkbbh * (\widehat{ |g_J|^{p-2k} } 1_{\Box})).
\end{align}
We will prove in Lemma~\ref{lem:counting} below that $\# S(\bar{K}_1,\dots,\bar{K}_k,\Box) \leq (q\kappa)^{-k(k-1)}$. 
If we think of the model case when $\kappa = 1/q$, this would say that the $\bar{K}_i$ and $\Box$ uniquely determine the $K_i$ in
\eqref{momentgeom}. This analogous to the situation in Linnik's lemma where once we upgrade \eqref{linnik}
to residue classes mod $p^k$, the remaining variables are essentially uniquely determined.

We now write
\begin{align*}
&\int_{\F^k}|g_{I_1} \dots g_{I_k}|^2 |g_J|^{p-2k} \\
&\quad= 
\Big| \sum_{\Box}  \sum_{\substack{\bar{K}_j \in P_{\delta}(I_j) \\ j=1,\dots,k}} \sum_{\substack{(K_1,\dots,K_k) \in \\ S(\bar{K}_1,\dots,\bar{K}_k,\Box)}} \widehat{g_{K_1}} * \dots * \widehat{g_{K_k}} * \gkbbh * (\widehat{ |g_J|^{p-2k} } 1_{\Box})(0) \Big|\\
&\quad\leq 
\sum_{\Box}  \sum_{\substack{\bar{K}_j \in P_{\delta}(I_j) \\ j=1,\dots,k}} \sum_{\substack{(K_1,\dots,K_k) \in \\ S(\bar{K}_1,\dots,\bar{K}_k,\Box)}} \int_{\F^k}|g_{K_1} \dots g_{K_k} g_{\bar{K}_1} \dots g_{\bar{K}_k}| |g_J|^{p-2k}*|\reallywidecheck{1_{\Box}}| \\
&\quad\leq 
\sum_{\Box}(\sum_{\bar{K} \in P_{\delta}} \|g_{\bar{K}}\|_{L^{\infty}(\F^k)})^k (q\kappa)^{-k(k-1)} \max_{K \in P_{\delta}} \|g_K\|_{L^{\infty}(\F^k)}^k \int_{\F^k}|g_J|^{p-2k}*|\reallywidecheck{1_{\Box}}|. 
\end{align*}
Since $\int_{\F^k}|\reallywidecheck{1_{\Box}}| = 1$ and the number of $\Box$ is $\nu^{-k(k - 1)/2}$, this gives
\begin{align*}
\int_{\F^k}|g_{I_1} &\dots g_{I_k}|^2 |g_J|^{p-2k} \leq 
\nu^{-\frac{k(k-1)}{2}} (q\kappa)^{-k(k-1)} \max_{K \in P_{\delta}} \|g_K\|_{L^{\infty}(\F^k)}^k (\sum_{\bar{K} \in P_{\delta}} \|g_{\bar{K}}\|_{L^{\infty}(\F^k)})^k  \int_{\F^k}|g_J|^{p-2k}.
\end{align*}
Applying affine rescaling shows that this is
\begin{align}
\begin{aligned}
\leq 
\nu^{-\frac{k(k-1)}{2}} (q\kappa)^{-k(k-1)}  &\max_{K \in P_{\delta}} \|g_K\|_{L^{\infty}(\F^k)}^k \times\\
&(\sum_{\bar{K} \in P_{\delta}} \|g_{\bar{K}}\|_{L^{\infty}(\F^k)})^k \mathfrak{D}_{p-2k}(\frac{\delta}{\nu})^{p-2k} \Big( \sum_{K' \in P_{\delta}(J)} \|g_{K'}\|_{L^{p-2k}(\F^k)}^2 \Big)^{\frac{p-2k}{2}}. \label{eq:analog_step3}
\end{aligned}
\end{align}
One can think of \eqref{eq:analog_step3} as the analogue of \eqref{eq:conclusion3.3} in Section~\ref{sect:3.3} in the following way:
the term $\nu^{-k(k - 1)/2}(q\kappa)^{-k(k - 1)}\max_{K \in P_{\delta}}\nms{g_K}_{\infty}^{k}$ plays the role of $p^{k(k - 1)/2}$ from Linnik's lemma, the term $(\sum_{\bar{K} \in P_{\delta}} \|g_{\bar{K}}\|_{\infty})^k$ plays
the role of $X^k$, and finally the term $\mathfrak{D}_{p - 2k}(\frac{\delta}{\nu})^{p - 2k}(\sum_{K' \in P_{\delta}(J)}\nms{g_{K'}}_{p - 2k}^{2})^{\frac{p - 2k}{2}}$
plays the role of the $J_{s - k, k}(X/p)$.

Plugging \eqref{eq:analog_step3} back to \eqref{eq:analog_step2}, we then obtain
\begin{align*}
    \int_{\F^k}|g|^p \lesssim_{p} &\mathfrak{D}_p(\frac{\delta}{\kappa})^p \Big( \sum_{K \in P_{\delta}} \|g_K\|_{L^{p}(\F^k)}^2 \Big)^{p/2} 
+ q^{-k(k-1)} \kappa^{-(k^2+4k-2)} \nu^{-\frac{k(k-1)}{2}} N^{p-2k}  \times \\
&\mathfrak{D}_{p-2k}(\frac{\delta}{\nu})^{p-2k} \max_{K \in P_{\delta}} \|g_K\|_{L^{\infty}(\F^k)}^k (\sum_{\bar{K} \in P_{\delta}} \|g_{\bar{K}}\|_{L^{\infty}(\F^k)})^k \Big( \sum_{K' \in P_{\delta}(J)} \|g_{K'}\|_{L^{p - 2k}(\F^k)}^2 \Big)^{\frac{p-2k}{2}}.
\end{align*}

\subsubsection{Geometry of the moment curve}
The proof of Lemma \ref{lem:main} is now complete modulo the proof of the following lemma, which  provides the key geometric input that enables one to count $\# S(\bar{K}_1,\dots,\bar{K}_k,\Box)$.
This is the analogue of Linnik's Lemma (\cite[Corollary 17]{Tao254A} and the estimate for $\mathbf{B(g)}$ in the proof of \cite[Lemma 5.1]{V}); see also \cite[Proposition 1.3]{GGPRY} and \cite[Proposition 3.1]{BBH}.
Both proofs use the Newton-Girard identities in essentially the same way. The hypothesis that $q > k$, where $q$ is the characteristic of our base field $\F$ and $k$ is the degree of the moment curve, plays a role in the following lemma.

\begin{lemma} \label{lem:counting}
Let $p \in 2k + 2\N$, $\delta \in q^{-\N}$, $\kappa \in q^{-\N} \cap [\delta,1)$, and $\nu = q^{\lfloor \log_q \delta^{1/k} \rfloor} \in q^{-\N}$ so that $\nu \leq \delta^{1/k}$. 
Suppose that $I_1, \dots, I_k \in P_{\kappa}$ with $d(I_i,I_j) > \kappa$ for all $i \ne j$. 
Let $\Box$ be a cube of side length $\nu^k$ and $\bar{K}_1 \in P_{\delta}(I_1), \ldots, \bar{K}_k \in P_{\delta}(I_k)$. Define  $S(\bar{K}_1,\dots,\bar{K}_k,\Box)$ be the set of all ordered $k$-tuples $(K_1,\dots,K_k)$ with $K_i \in P_{\delta}(I_i)$ such that 
$$0 \in \supp (\widehat{g_{K_1}} * \dots * \widehat{g_{K_k}} * \gkbbh * (\widehat{ |g_J|^{p-2k} } 1_{\Box})).$$
Then
\begin{equation*}
\# S(\bar{K}_1,\dots,\bar{K}_k,\Box) \leq (q\kappa)^{-k(k-1)}
.\end{equation*}
\end{lemma}

\begin{proof}
Assume for the sake of contradiction that $\#S(\bar{K}_1, \ldots, \bar{K}_{k}, \Box) > (q\kappa)^{-k(k - 1)} \geq 1$. We can find two 
$k$-tuples of intervals $(A_{1}, \ldots, A_k)$ and $(B_{1}, \ldots, B_k)$ with each $A_{i}, B_{i} \in P_{\delta}(I_{i})$ such that
\begin{align}
0 &\in \supp (\widehat{g_{A_1}} * \dots * \widehat{g_{A_k}} * \gkbbh * (\widehat{ |g_J|^{p-2k} } 1_{\Box})),\label{eq1v1}\\
0 &\in \supp (\widehat{g_{B_1}} * \dots * \widehat{g_{B_k}} * \gkbbh * (\widehat{ |g_J|^{p-2k} } 1_{\Box})),\label{eq2v1}
\end{align}
and such that there exists an $i_0$ with $d(A_{i_0}, B_{i_0}) > (q\kappa)^{-(k - 1)}\delta$. Indeed, if not, picking an arbitrary $(C_1, \ldots, C_k) \in S(\bar{K}_1, \ldots, \bar{K}_{k}, \Box)$, 
shows that any other $(D_1, \ldots, D_k) \in S(\bar{K}_1, \ldots, \bar{K}_{k}, \Box)$ must satisfy $d(C_i, D_i) \leq (q\kappa)^{-(k - 1)}\delta$. This gives at most
$(q\kappa)^{-k(k - 1)}$ many $k$-tuples which violates our initial assumption that $\#S(\bar{K}_1, \ldots, \bar{K}_{k}, \Box) > (q\kappa)^{-k(k - 1)}$.
Without loss of generality, we may assume that $i_0 = 1$.

Since for each $i = 1, 2, \ldots, k$, we have $A_{i}, B_{i} \subset I_{i}$ and $d(I_{i}, I_{j}) > \kappa$ for all $i \neq j$, this implies
\begin{align}\label{distv1}
d(A_{i}, A_{j}) \geq q\kappa, \quad d(B_{i}, B_{j}) \geq q\kappa, \quad d(A_{i}, B_j) \geq q\kappa \quad \text{whenever $j \neq i$}
\end{align}
(thus the only distances we do not have any control over are the ones of the form $d(A_i, B_i)$, $i \neq 1$). By \eqref{eq1v1} and \eqref{eq2v1}, we have that
\begin{align*}
0 \in (\sum_{i = 1}^{k}\tau_{A_i} - \sum_{i = 1}^{k}\tau_{\bar{K}_i} + \Box) \cap (\sum_{i = 1}^{k}\tau_{B_i} - \sum_{i = 1}^{k}\tau_{\bar{K}_i} + \Box)
\end{align*}
where here we recall the definition of $\tau_K$ in \eqref{taukdef}.
Each $\tau_{A_i}$, $\tau_{B_i}$, and $\tau_{\bar{K}_i}$ are cubes in $\F^k$ of side length $\delta$ and $\Box$ is a cube in $\F^k$
of side length $\nu^{k} \leq \delta$. Thus by the ultrametric inequality, both $\sum_{i = 1}^{k}\tau_{A_i} - \sum_{i = 1}^{k}\tau_{\bar{K}_i} + \Box$
and $\sum_{i = 1}^{k}\tau_{B_i} - \sum_{i = 1}^{k}\tau_{\bar{K}_i} + \Box$ are cubes in $\F^k$ of side length $\delta$. Furthermore, by the ultrametric inequality,
since two cubes of side length $\delta$ are either completely disjoint or exactly the same, we must have
\begin{align*}
\sum_{i = 1}^{k}\tau_{A_i} - \sum_{i = 1}^{k}\tau_{\bar{K}_i} + \Box = \sum_{i = 1}^{k}\tau_{B_i} - \sum_{i = 1}^{k}\tau_{\bar{K}_i} + \Box
\end{align*}
and hence
\begin{align*}
\sum_{i = 1}^{k}\tau_{A_i} - \sum_{i = 1}^{k}\tau_{B_i} = B(0, \delta).
\end{align*}
Therefore (after another application of the ultrametric inequality) 
there exists $\xi_{A_i} \in A_i$ and $\xi_{B_i} \in B_i$ such that
\begin{align}\label{power}
|\sum_{i = 1}^{k}\xi_{A_i}^{j} - \sum_{i = 1}^{k}\xi_{B_i}^{j}| \leq \delta
\end{align}
for $j = 1, 2, \ldots, k$.

We now use the Newton-Girard identities to derive a contradiction.
For $j = 1, 2, \ldots, k$, define the power sums
$p_{j}(x_1, \ldots, x_k) := x_{1}^{j} + \cdots + x_{k}^{j}$.
Next for $j = 1, 2, \ldots, k$, define the elementary symmetric polynomials
$e_{j}(x_1, \ldots, x_k) := \sum_{1 \leq i_1 < \cdots < i_j \leq k}x_{i_1} \cdots x_{i_j}$. Additionally, let $e_{0}(x_1, \ldots, x_k) := 1$.
Then we have the two identities:
\begin{align}\label{expand}
(X - x_1)(X - x_2) \cdots (X - x_k) = \sum_{j = 0}^{k}(-1)^{j}e_{j}(x_1, \ldots, x_k)X^{k - j}
\end{align}
and for $j = 1, 2, \ldots, k$, we have
\begin{align}\label{newton}
je_{j}(x_1, \ldots, x_k) = \sum_{i = 0}^{j - 1}(-1)^{i}e_{j - i - 1}(x_1, \ldots, x_k)p_{i + 1}(x_1, \ldots, x_k).
\end{align}
See, for example, \cite[Lemma 15]{Tao254A} for a proof.

Let $e_{j}(A) := e_{j}(\xi_{A_1}, \ldots, \xi_{A_k})$
and $p_{j}(A) := p_{j}(\xi_{A_1}, \ldots, \xi_{A_k})$. Similarly
define $e_{j}(B)$ and $p_{j}(B)$.
By \eqref{expand}, we then have
\begin{align}\label{new1}
(\xi_{A_1} - \xi_{B_1}) \cdots (\xi_{A_1} - \xi_{B_k}) = \sum_{j = 0}^{k}(-1)^{j}e_{j}(B)\xi_{A_{1}}^{k - j}
\end{align}
and
\begin{align}\label{new2}
0 = (\xi_{A_1} - \xi_{A_1}) \cdots (\xi_{A_1} - \xi_{A_k}) = \sum_{j = 0}^{k}(-1)^{j}e_{j}(A)\xi_{A_1}^{k - j}.
\end{align}
Subtracting \eqref{new2} from \eqref{new1} and using that $|\xi_{A_1} - \xi_{B_j}| \geq q\kappa$
for any $j \neq 1$ (which follows from \eqref{distv1}) shows that
\begin{align}\label{difftarget}
(q\kappa)^{k - 1}|\xi_{A_1} - \xi_{B_1}| \leq |\sum_{j = 0}^{k}(-1)^{j}(e_{j}(B) - e_{j}(A))\xi_{A_1}^{k - j}| \leq \max_{j}|e_{j}(B) - e_{j}(A)|.
\end{align}
Next we claim that $|e_{j}(B) - e_{j}(A)| \leq \delta$ for all $j = 1, 2, \ldots, k$.
We prove this by induction. Since $e_{1} = p_{1}$, $|e_{1}(B) - e_{1}(A)| \leq \delta$
by the $j = 1$ case of \eqref{power}. Now assume that for some $J = 1, 2, \ldots, k - 1$ we had $|e_{j}(B) - e_{j}(A)| \leq \delta$
for all $j = 1, 2, \ldots, J$. Then by \eqref{newton},
\begin{align*}
|(J+1)e_{J + 1}(B) - (J+1)e_{J + 1}(A)| &= |\sum_{i = 0}^{J}(-1)^{i}(e_{J - i}(B)p_{i + 1}(B) - e_{J - i}(A)p_{i + 1}(A))|\\
& \leq \max_{0 \leq i \leq J}|e_{J - i}(B)p_{i + 1}(B) - e_{J - i}(A)p_{i + 1}(A)|.
\end{align*}
Observe that
\begin{align*}
|e_{J - i}(B)p_{i + 1}(B) - e_{J - i}(A)p_{i + 1}(A)|& = |e_{J - i}(B)(p_{i + 1}(B) - p_{i + 1}(A)) + p_{i + 1}(A)(e_{J - i}(B) - e_{J - i}(A))|\\
&\leq \max(|p_{i + 1}(B) - p_{i + 1}(A)|, |e_{J - i}(B) - e_{J - i}(A)|) \leq \delta
\end{align*}
by the inductive hypothesis and \eqref{power}. 
Since $q$ is a prime $> k$, it follows that $|e_{J + 1}(B) - e_{J + 1}(A)| \leq \delta$.

Applying this conclusion to \eqref{difftarget} then yields that
$(q\kappa)^{k - 1}|\xi_{A_1} - \xi_{B_1}| \leq \delta.$
But this contradicts the fact that $d(A_1, B_1) > (q\kappa)^{-(k - 1)}\delta$.
Therefore we must have $\#S(\bar{K}_1, \ldots, \bar{K}_{k}, \Box) \leq (q\kappa)^{-k(k - 1)}$ which completes the proof of the lemma.
\end{proof}

\section{Proof of Theorem \ref{thm:main} and Corollary \ref{maincor}}\label{mainproof}
\subsection{Dyadic pigeonholing}\label{pigeon}
It is more convenient to bound 
\begin{align}\label{supdec}
D_p(\delta) := \sup_{\delta_0 \in q^{-\N} \cap [\delta,1]} \mathfrak{D}_p(\delta_0)
\end{align}
instead of $\mathfrak{D}_p(\delta)$ as $D_{p}(\delta)$ is defined for all real $\delta \in (0, 1]$ (rather than just for $\delta \in q^{-\N}$) and is monotonic, that is, $D_{p}(\delta_L) \leq D_{p}(\delta_S)$ if $\delta_L \geq \delta_S$.

\begin{prop}
For even integers $p > 2k$, there exists a constant $C > 0$, depending only on $k$ and $p$, such that for every $0 < \varepsilon < 1$, we have
\begin{equation} \label{eq:D_iterate}
D_p(\delta)^p \leq C (\log \delta^{-1})^{3p} \left[ D_p(\delta^{1-\varepsilon})^p + q^{\frac{p}{2} + \frac{k^2 + 7k - 4}{2}}\delta^{-(k^2+4k-2) \varepsilon} \delta^{-\frac{1}{k}(\frac{p}{2}+\frac{k(k-3)}{2})} D_{p-2k}(\delta^{1-\frac{1}{k}})^{p-2k} \right]
\end{equation}
for all $0 < \delta < 1$.
\end{prop}

\begin{proof}
To bound $D_p(\delta)^p$, suppose $0 < \delta < 1$ and $\delta_0 \in q^{\Z}$ with $\delta_0 \in [\delta,1]$. We need to bound $\mathfrak{D}_p(\delta_0)^p$ by decoupling down to frequency scale $\delta_0$. 

Let $f$ be a Schwartz function on $\F^k$ with Fourier support in $\bigcup_{K \in P_{\delta_0}} \theta_K$. Then $f = \sum_{K \in P_{\delta_0}} f_K$ where $\widehat{f_K} := \widehat{f} 1_{K \times \F^{k-1}}$. 
We want to prove the existence of $C > 0$ so that for any $0 < \vep < 1$,
\begin{equation} \label{eq:f_general}
\begin{split}
\int_{\F^k}&|f|^p \leq  C (\log \delta^{-1})^{3p}\times 
\\ & \left[ D_p(\delta^{1-\varepsilon})^p + q^{\frac{p}{2} + \frac{k^2 + 7k - 4}{2}}\delta^{-(k^2+4k-2) \varepsilon} \delta^{-\frac{1}{k}(\frac{p}{2}+\frac{k(k-3)}{2})} D_{p-2k}(\delta^{1-\frac{1}{k}})^{p-2k} \right]\Big( \sum_{K \in P_{\delta_0}} \|f_K\|_{L^{p}(\F^k)}^2 \Big)^{p/2}.
\end{split}
\end{equation}
In fact, we will prove that for any translate $Q$ of $B_{\delta_0^{-k}} := \{x \in \Q_q^k \colon |x| \leq \delta_0^{-k}\}$, we have
\begin{equation} \label{eq:f_compsupp}
\begin{split}
\int_{Q}&|f|^p \leq  C (\log \delta^{-1})^{3p} \times
\\ & \left[ D_p(\delta^{1-\varepsilon})^p + q^{\frac{p}{2} + \frac{k^2 + 7k - 4}{2}}\delta^{-(k^2+4k-2) \varepsilon} \delta^{-\frac{1}{k}(\frac{p}{2}+\frac{k(k-3)}{2})} D_{p-2k}(\delta^{1-\frac{1}{k}})^{p-2k} \right]\Big( \sum_{K \in P_{\delta_0}} \|f_K\|_{L^p(Q)}^2 \Big)^{p/2}.
\end{split}
\end{equation}
The estimate \eqref{eq:f_general} then follows by summing over all such $Q$'s that tile $\F^k$, and applying Minkowski's inequality to bring an $\ell^{p/2}$ norm over $Q$ on the right hand side into the sum over $K \in P_{\delta_0}$.

Thus we now turn to the proof of \eqref{eq:f_compsupp}. Note that for any translate $Q$ of $B_{\delta_0^{-k}}$, we have that
$\wh{1}_Q$ is supported in $B_{\delta_{0}^k}$. Therefore $f1_Q$ is still Fourier supported in $\bigcup_{K \in P_{\delta_0}}\ta_K$
since $\ta_K + B_{\delta_{0}^k} = \ta_K$ for all $K \in P_{\delta_0}$. 
Next, we have $(f 1_Q)_K = f_K 1_Q$; indeed 
\[
\widehat{(f 1_Q)_K} = \widehat{f 1_Q} 1_{K \times \F^{k-1}} = (\widehat{f} * \widehat{1_Q}) 1_{K \times \F^{k-1}} = (\widehat{f} 1_{K \times \F^{k-1}}) * \widehat{1_Q} = \widehat{f_K} * \widehat{1_Q}.
\]
As a result, to prove \eqref{eq:f_compsupp}, it suffices to prove \eqref{eq:f_general} under the additional assumption that $f$ is supported on $Q$. 
Since $f$ is an arbitrary Schwartz function with Fourier support in $\bigcup_{K \in P_{\delta_0}}\ta_K$, we may assume $Q = B_{\delta_0^{-k}}$. Thus from now on, we assume additionally that $f$ and all the $f_K$ are supported on $B_{\delta_0^{-k}}$ and prove \eqref{eq:f_general}.

We first dyadically pigeonhole $f$ by wavepacket height.
Write $H^* = \max_{K \in P_{\delta_0}} \|f_K\|_{L^{\infty}(\F^k)}$. For $K \in P_{\delta_0}$ and $H \in 2^{\Z}H^{\ast} \cap (\delta_0^{1+\frac{k(k-1)}{2p}} H^*, H^*]$, let
\[
f^{(H)}_K = f_K 1_{H/2 < |f_K| \leq H}
\]
where here $f_K: \F^k \rightarrow \C$, $|f_K|$ is the absolute value of $f_K$, and the last characteristic function is meant to be the indicator function of the set $\{x \in \F^k: H/2 < |f_{K}(x)| \leq H\}$.
Since $f_{K}$ is supported on $B_{\delta_{0}^{-k}}$, so is $f^{(H)}_{K}$.
By Lemma \ref{wavepacket}, since $f_K$ is Fourier supported in $\ta_K$, we
then have
\begin{align}\label{fhkdef}
f^{(H)}_{K} = (\sum_{T \in \T(K)}f_{K}1_{T})1_{H/2 < |f_K| \leq H} = \sum_{T \in \T(K)}(f_{K}1_{T})1_{H/2 < |f_{K}1_T| \leq H}
\end{align}
where the last equality is because $|f_{K}1_T|$ constant on every $T \in \T(K)$.
Again by Lemma \ref{wavepacket}, note that $f^{(H)}_{K}$
is Fourier supported in $\ta_K$.
Using the terminology of Lemma \ref{wavepacket}, the nonzero wavepackets
that make up $f^{(H)}_{K}$ are all of height $\sim H$.
Then
\begin{align*}
\|f-\sum_{K \in P_{\delta_0}} \sum_{H \in 2^{\Z}H^{\ast} \cap (\delta_0^{1+\frac{k(k-1)}{2p}} H^*, H^*]} f^{(H)}_K\|_{L^{\infty}(\F^k)}
&\leq \sum_{K \in P_{\delta_0}} \|f_K 1_{|f_K| \leq \delta_0^{1+\frac{k(k-1)}{2p}} H^*}\|_{L^{\infty}(\F^k)} \\
&\leq \delta_0^{-1} (\delta_0^{1+\frac{k(k-1)}{2p}} H^*) = \delta_0^{\frac{k(k-1)}{2p}} H^*
\end{align*}
so since $f$ and $f_{K}^{(H)}$ are supported on $B_{\delta_{0}^{-k}}$,
\begin{align*}
\|f-\sum_{K \in P_{\delta_0}} \sum_{H \in 2^{\Z}H^{\ast} \cap (\delta_0^{1+\frac{k(k-1)}{2p}} H^*, H^*]} f^{(H)}_K\|_{L^{p}(\F^k)} 
& \leq (\delta_0^{\frac{k(k-1)}{2p}} H^*) |B_{\delta_0^{-k}}|^{\frac{1}{p}} = H^* \delta_0^{-\frac{k(k+1)}{2p}} \\ & \leq \max_{K \in P_{\delta_0}} \|f_K\|_{L^{p}(\F^k)}
\leq (\sum_{K \in P_{\delta_0}} \|f_K\|_{L^{p}(\F^k)}^2)^{1/2}
\end{align*}
where the second inequality follows from writing $f_{K} = f_{K} \ast \reallywidecheck{1}_{\ta_K}$ and applying Young's inequality $\nms{f_K}_{L^{\infty}(\F^k)} \leq \nms{f_K}_{L^{p}(\F^k)}\nms{\reallywidecheck{1}_{\ta_K}}_{L^{p'}(\F^k)}$.
This shows
\[
\|f\|_{L^{p}(\F^k)} \leq \sum_{H \in 2^{\Z}H^{\ast} \cap (\delta_0^{1+\frac{k(k-1)}{2p}} H^*, H^*]} \|\sum_{K \in P_{\delta_0}} f^{(H)}_K \|_{L^{p}(\F^k)} + (\sum_{K \in P_{\delta_0}} \|f_K\|_{L^{p}(\F^k)}^2)^{1/2}.
\]

Next we dyadically pigeonhole so that each relevant $f_{K}^{(H)}$ is made
up of about the same number of wavepackets.
Let now $\nu = q^{\lfloor \log_q \delta_0^{1/k}\rfloor} \leq \delta_0^{1/k}$. 
From \eqref{fhkdef}, $f^{(H)}_{K}$ is Fourier supported in $\ta_K$
and supported in $B_{\delta_{0}^{-k}}$.
Since a $T \in \T(K)$ is either completely contained in or
completely disjoint from $B_{\delta_{0}^{-k}}$, 
we then can write
\begin{align}\label{fhkdef2}
f^{(H)}_{K} = \sum_{T \in \T(K), T \subset B_{\delta_{0}^{-k}}}(f_{K}1_T)1_{H/2 < |f_{K}1_T| \leq H}.
\end{align}
Furthermore, the $T \in \T(K)$ which are contained in $B_{\delta_{0}^{-k}}$
perfectly partition $B_{\delta_{0}^{-k}}$ into $\delta_{0}^{-k(k - 1)/2}$
many translates of $T_{0, k}$. 
Thus \eqref{fhkdef2} has at most $\delta_{0}^{-k(k - 1)/2}$ many nonzero terms.
Therefore
for $\alpha \in 2^{\N} \cap [1,\delta_0^{-k(k-1)/2}]$, let
\[
f^{(H,\alpha)}_K := f^{(H)}_K
\]
if the number of nonzero terms in \eqref{fhkdef2} (that
is, the number of nonzero wavepackets in $f^{(H)}_{K}$) is in $(\alpha/2, \alpha]$, and $0$ otherwise.
Thus now we have that
\begin{align}\label{numberdecompose}
f_{K}^{(H)} = \sum_{\alpha \in 2^{\N} \cap [1, \delta_{0}^{-k(k - 1)/2}]}f_{K}^{(H, \alpha)}
\end{align}
and each $f_{K}^{(H, \alpha)}$ is a function which is supported in $B_{\delta_{0}^{-k}}$
and Fourier supported in $\ta_K$ which has $\sim \alpha$
many nonzero wavepackets of height $\sim H$.

Finally, we dyadically pigeonhole so that given a $K$, the parent interval $J$
of length $\nu$ has about the same number of children $K'$ of length $\delta_0$
such that $f_{K'}^{(H, \alpha)} \neq 0$. 
To be more precise, fix a $K$ and let $J$ be the unique parent interval of length $\nu$
containing $K$.
%This parent $J$ has a total of $\nu/\delta_0$ many children (intervals) $K'$ of length $\delta_0$.
This parent $J$ contains $\nu/\delta_0$ many intervals $K'$ of length $\delta_0$
and hence $J$ has at most $\nu/\delta_0$ many children $K'$ such that $f_{K'}^{(H, \alpha)} \neq 0$.
%We call these $K'$, the siblings of $K$ (with respect to $J$). 
%Thus for each $K$, it has at most $(\nu/\delta_0) - 1$ many siblings $K'$ such that $f_{K'}^{(H, \alpha)} \neq 0$. 
For $K \subset J$ and $\beta \in 2^{\N} \cap [1, \nu/\delta_0]$,
let
\[
f^{(H,\alpha,\beta)}_K := f^{(H,\alpha)}_K
\]
if the number of children $K''$ of $J$ with $f^{(H,\alpha)}_{K''} \ne 0$ is in $(\beta/2,\beta]$, that is, if $\#\{K'' \in P_{\delta_0}(J) : f_{K''}^{(H, \alpha)} \neq 0\} \in (\beta/2, \beta]$, and 0 otherwise. 
Thus we now have
\begin{align}\label{siblingdecompose}
f_{K}^{(H, \alpha)} = \sum_{\beta \in 2^{\N} \cap [1, \nu/\delta_0]}f_{K}^{(H, \alpha, \beta)}
\end{align}
and each $f_{K}^{(H, \alpha, \beta)}$ is a function which is supported in $B_{\delta_{0}^{-k}}$, Fourier supported in $\ta_K$, has $\sim \alpha$ many nonzero wavepackets
of height $\sim H$, and $K$'s parent $J$ has $\sim \beta$ children each of which
also are supported in $B_{\delta_{0}^{-k}}$, Fourier supported in $\ta_K$,
and have $\sim \alpha$ many nonzero wavepackets of height $\sim H$.

Thus combining \eqref{numberdecompose} and \eqref{siblingdecompose} gives
\[
f^{(H)}_K =  \sum_{\alpha \in 2^{\N} \cap [1,\delta_0^{-k(k-1)/2}]} \sum_{\beta \in 2^{\N} \cap [1, \nu/\delta_0]} f^{(H,\alpha,\beta)}_K 
\]
which implies
\begin{align*}
\|f\|_{L^{p}(\F^k)} \leq \sum_{H \in 2^{\Z}H^{\ast} \cap (\delta_0^{1+\frac{k(k-1)}{2p}} H^*, H^*]} \sum_{\alpha \in 2^{\N} \cap [1,\delta_0^{-k(k-1)/2}]} \sum_{\beta \in 2^{\N} \cap [1, \nu/\delta_0]} &\|\sum_{K \in P_{\delta_0}} f^{(H,\alpha,\beta)}_K \|_{L^{p}(\F^k)}\\
&\quad\quad + (\sum_{K \in P_{\delta_0}} \|f_K\|_{L^{p}(\F^k)}^2)^{1/2}.
\end{align*}

Fix now $\varepsilon > 0$. For each of the $\lsm (\log \delta_{0}^{-1})^3$ choices of $(H,\alpha,\beta)$, we apply Lemma~\ref{lem:main} with $\delta$ replaced by $\delta_0$, to 
\begin{align}\label{gexplicit}
g := \sum_{K \in P_{\delta_0}} f^{(H,\alpha,\beta)}_K = \sum_{J \in P_{\nu}^{(H, \alpha, \beta)}}\sum_{K \in P_{\delta_0}(J)}f_{K}^{(H, \alpha)}
\end{align}
where
\begin{align*}
P_{\nu}^{(H, \alpha, \beta)} = \{J \in P_{\nu} : \#\{K'' \in P_{\delta_0}(J) : f_{K''}^{(H, \alpha)} \neq 0\} \in (\beta/2, \beta]\}
\end{align*}
and $\kappa := q^{\lfloor \log_q \delta_0^{\varepsilon} \rfloor} \leq \delta_0^{\varepsilon}$.
Note that this implies
\begin{align}\label{gjexplicit}
g_{J} = 1_{P_{\nu}^{(H, \alpha, \beta)}}(J)\sum_{K \in P_{\delta_0}(J)}f_{K}^{(H, \alpha)}
\end{align}
and $g_K = f_{K}^{(H, \alpha)}$ if $K$'s parent $J$ is contained in $P_{\nu}^{(H, \alpha, \beta)}$ and 0 otherwise.
Write $N$ for the number of $J \in P_{\nu}$ for which $g_J \ne 0$ as in Lemma~\ref{lem:main}, and so $N = \# P_{\nu}^{(H, \alpha, \beta)}$.
Note that by assumption the number of nonzero terms in the $\sum_{K}$ in \eqref{gjexplicit} is $\sim \beta$.

%Fix now $\varepsilon > 0$. For each of the $\lsm (\log \delta_{0}^{-1})^3$ choices of $(H,\alpha,\beta)$, we apply Lemma~\ref{lem:main} with $\delta$ replaced by $\delta_0$, to 
%\begin{align}\label{gexplicit}
%g := \sum_{K \in P_{\delta_0}} f^{(H,\alpha,\beta)}_K = \sum_{\st{J \in P_{\nu}\\\sum_{K'' \in P_{\delta_0}(J)}1_{f_{K''}^{(H, \alpha)} \neq 0}(K'') \in (\beta/2, \beta]}}\sum_{K \in P_{\delta_0}(J)}f_{K}^{(H, \alpha)}
%\end{align}
%and $\kappa := q^{\lfloor \log_q \delta_0^{\varepsilon} \rfloor} \leq \delta_0^{\varepsilon}$.
%Note that this implies
%\begin{align}\label{gjexplicit}
%g_{J} = 1_{\sum_{K'' \in P_{\delta_0}(J)}1_{f_{K''}^{(H, \alpha)} \neq 0}(K'') \in (\beta/2, \beta]}(J)\sum_{K \in P_{\delta_0}(J)}f_{K}^{(H, \alpha)}
%\end{align}
%and $g_K = f_{K}^{(H, \alpha)}$ if $K$'s parent $J$ is such that $\sum_{K'' \in P_{\delta_0}(J)}1_{f_{K''}^{(H, \alpha)} \neq 0}(K'') \in (\beta/2, \beta]$.
%Write $N$ for the number of $J \in P_{\nu}$ for which $g_J \ne 0$ as in Lemma~\ref{lem:main}, that is, $N$ is the number of nonzero terms in the $\sum_{J}$ in \eqref{gexplicit}.
%Note that by assumption the number of nonzero terms in the $\sum_{K}$ in \eqref{gjexplicit} is $\sim \beta$. 

With this, we then first compute
\begin{align}\label{pig1}
\max_{K \in P_{\delta_0}}\nms{g_K}_{L^{\infty}(\F^k)}^{k} \sim H^{k}
\end{align}
since $g_K = f^{(H, \alpha, \beta)}_{K}$ which has height $\sim H$.
Next, we have
\begin{align}
\begin{aligned}\label{pig2}
(\sum_{\bar{K} \in P_{\delta_0}}\nms{g_{\bar{K}}}_{L^{\infty}(\F^k)})^{k} &= (\sum_{J \in P_{\nu}}\sum_{\bar{K} \in P_{\delta_0}(J)}\nms{g_{\bar{K}}}_{L^{\infty}(\F^k)})^{k}\\
&= (\sum_{J \in P_{\nu}^{(H, \alpha, \beta)}}\sum_{\bar{K} \in P_{\delta_0}(J)}\nms{g_{\bar{K}}}_{L^{\infty}(\F^k)})^{k} \sim (N\beta H)^{k}
\end{aligned}
\end{align}
since there are $N$ such $J$ for which $g_{J} \neq 0$ and by how $g$ is defined, each of these $J$'s that contribute has $\sim \beta$ children $\bar{K}$ such that $g_{\bar{K}} = f_{\bar{K}}^{(H, \alpha)} \neq 0$. We can finish this estimate once again by using that
$g_{\bar{K}}$ has height $\sim H$.
Third,
\begin{align}\label{pig3}
\max_{J \in P_{\nu}}(\sum_{K \in P_{\delta_0}(J)}\nms{g_K}_{L^{p - 2k}(\F^k)}^{2})^{(p - 2k)/2} &= \max_{J \in P_{\nu}}(\sum_{K \in P_{\delta_0}(J)}\nms{g_K}_{L^{p - 2k}(B_{\delta_{0}^{-k}})}^{2})^{(p - 2k)/2}\nonumber\\
&\sim_{p, k} \beta^{(p - 2k)/2}H^{p - 2k}\alpha \delta_{0}^{-k(k + 1)/2}
\end{align}
since by how $g$ is defined, the $\sum_{K \in P_{\delta_0}(J)}$
has $\sim \beta$ terms and each term is made up of $\sim \alpha$ wavepackets
of height $\sim H$. Note here we made use that each $T \in \T(K)$ has volume
$\delta_{0}^{-k(k + 1)/2}$ and $g_K$ is supported on $B_{\delta_{0}^{-k}}$.
Finally, a similar computation gives that
\begin{align}\label{pig4}
(\sum_{K \in P_{\delta_0}}\nms{g_K}_{L^{p}(\F^k)}^{2})^{p/2} = (\sum_{J \in P_{\nu}}\sum_{K \in P_{\delta_0}(J)}\nms{g_K}_{L^{p}(B_{\delta_{0}^{-k}})}^{2})^{p/2} \sim_{p, k} (N\beta)^{p/2}H^{p}\alpha\delta_{0}^{-k(k + 1)/2}.
\end{align}
Combining \eqref{pig1}-\eqref{pig4} gives that
\begin{align*}
\max_{K \in P_{\delta_0}}\nms{g_K}_{L^{\infty}(\F^k)}^{k}(\sum_{\bar{K} \in P_{\delta_0}}\nms{g_{\bar{K}}}_{L^{\infty}(\F^k)})^{k}\max_{J \in P_{\nu}}(\sum_{K \in P_{\delta_0}(J)}&\nms{g_K}_{L^{p - 2k}(\F^k)}^{2})^{(p - 2k)/2}\\ &\sim_{p, k} N^{-(p - 2k)/2}(\sum_{K \in P_{\delta_0}}\nms{g_K}_{L^{p}(\F^k)}^{2})^{p/2}.
\end{align*}
Using this with Lemma \ref{lem:main} where $g$ is as given in \eqref{gexplicit}, then shows that
\begin{align*}
\int_{\F^k}|g|^p 
\leq \,  C \mathfrak{D}_p(\frac{\delta_0}{\kappa})^p &\Big( \sum_{K \in P_{\delta_0}} \|g_K\|_{L^{p}(\F^k)}^2 \Big)^{p/2} \\ & + C q^{-k(k - 1)}\kappa^{-(k^2+4k-2)} \nu^{-\frac{k(k-1)}{2}} N^{\frac{p-2k}{2}} \mathfrak{D}_{p-2k}(\frac{\delta_0}{\nu})^{p-2k} \Big( \sum_{K \in P_{\delta_0}} \|g_K\|_{L^{p}(\F^k)}^2 \Big)^{p/2}.
\end{align*}
Note $\frac{\delta_0}{\kappa} \geq \frac{\delta_0}{\delta_0^{\varepsilon}} = \delta_0^{1-\varepsilon} \geq \delta^{1-\varepsilon}$, so $\mathfrak{D}_{p}(\frac{\delta_0}{\kappa}) \leq D_{p}(\frac{\delta_0}{\kappa}) \leq D_{p}(\delta^{1-\varepsilon})$ where in the second inequality we have used monotonicity. Similarly, $\frac{\delta_0}{\nu} \geq \frac{\delta_0}{\delta_0^{1/k}} = \delta_0^{1-\frac{1}{k}} \geq \delta^{1-\frac{1}{k}}$, so $\mathfrak{D}_{p-2k}(\frac{\delta_0}{\nu}) \leq D_{p-2k}(\delta^{1-\frac{1}{k}})$.
As a result,
\begin{align*}
\int_{\F^k}|g|^p \leq \,  C &D_p(\delta^{1-\varepsilon})^p \Big( \sum_{K \in P_{\delta_0}} \|g_K\|_{L^{p}(\F^k)}^2 \Big)^{p/2} \\ & + C q^{-k(k - 1)}\kappa^{-(k^2+4k-2)} \nu^{-\frac{k(k-1)}{2}} N^{\frac{p-2k}{2}} D_{p-2k}(\delta^{1-\frac{1}{k}})^{p-2k} \Big( \sum_{K \in P_{\delta_0}} \|g_K\|_{L^{p}(\F^k)}^2 \Big)^{p/2}.
\end{align*}
Now use $N \leq \nu^{-1}$ and $\|g_K\|_{L^{p}(\F^k)} = \|f^{(H,\alpha,\beta)}_K\|_{L^{p}(\F^k)} \leq \|f_K\|_{L^{p}(\F^k)}$. Thus
\begin{align*}
\int_{\F^k}|f|^p&\leq C (\log \delta^{-1})^{3p} \times\\
 &\left[ D_p(\delta^{1-\varepsilon})^p + q^{-k(k - 1)}\kappa^{-(k^2+4k-2)} \nu^{-(\frac{p}{2}+\frac{k(k-3)}{2})} D_{p-2k}(\delta^{1-\frac{1}{k}})^{p-2k} \right] \Big( \sum_{K \in P_{\delta_0}} \|f_K\|_{L^{p}(\F^k)}^2 \Big)^{p/2}.
\end{align*}
But $\nu^{-1} \leq q \delta_0^{-1/k} \leq q \delta^{-1/k}$ and $\kappa^{-1} \leq q \delta_0^{-\varepsilon} \leq q \delta^{-\varepsilon}$. 
This completes the proof of \eqref{eq:D_iterate}.
\end{proof}

\subsection{Proof of Theorem~\ref{thm:main}}
We now finish the proof of Theorem \ref{thm:main}.
It suffices to iterate \eqref{eq:D_iterate} by using an induction on $p$ and induction on $\delta$. Applying the definition of $D_{p}(\delta)$ from \eqref{supdec} and the
hypothesis of Theorem \ref{thm:main} gives that
\begin{align*}
D_{p_0}(\delta)^{p_0} \leq C_{1}\delta^{-(\frac{p_0}{2} - \frac{k(k + 1)}{2}) - c(p_0)(1 - \frac{1}{k})^{p_{0}/(2k)}}
\end{align*}
for all $\delta \in (0, 1)$ and some $c(p_0) \geq 0$ such that the power of $\delta^{-1}$ is nonnegative.
Note that from \eqref{adef}, $a(p, p_0) = a(p - 2k, p_0) + \frac{p}{2} + \frac{k^2 + 7k - 4}{2}$ and $a(p_0, p_0) = 0$.
Additionally, \eqref{positiveexp} gives $b(p_0) \geq 0$ where
\[
b(p):= (p-k(k+1))(1 - \frac{1}{k})^{-\frac{p}{2k}} + 2 c(p_0),
\]
and $b(p)$ is an increasing function of $p$ on $[2,\infty)$: indeed, for $p \geq 2$, we have
\begin{align*}
 b'(p) 
 &= (1-\frac{1}{k})^{-\frac{p}{2k}} [ 1 + \frac{p-k(k+1)}{2k} \log(1-\frac{1}{k})^{-1}]\\
 &= (1-\frac{1}{k})^{-\frac{p}{2k}} [ 1 + \frac{p-k(k+1)}{2k} (\log k - \log(k-1))]\\
 &\geq (1-\frac{1}{k})^{-\frac{p}{2k}} [ 1 + \frac{2-k(k+1)}{2k} (\log k - \log(k-1))]\\
 &\geq (1-\frac{1}{k})^{-\frac{p}{2k}} [ 1 + \frac{2-k(k+1)}{2k} \frac{1}{k-1}],
\end{align*}
where we used $p \geq 2$ in the first inequality, and used $\log k - \log(k-1) \leq \frac{1}{k-1}$ with $2 - k(k+1) \leq 0$ in the second inequality.
This gives 
\[
b'(p) \geq (1-\frac{1}{k})^{-\frac{p}{2k}} [ 1 + \frac{2-k-k^2}{2k(k-1)}] = (1-\frac{1}{k})^{-\frac{p}{2k}} [ 1 - \frac{k+2}{2k}] \geq 0
\]
since $k \geq 2$, proving that $b(p)$ is an increasing function of $p$ on $[2,\infty)$. As a result, from $b(p_0) \geq 0$, we see that $b(p) \geq 0$ for all $p \geq p_0$, and hence 
\begin{align}\label{pineq}
\frac{p}{2} - \frac{k(k + 1)}{2} + c(p_0)(1 - \frac{1}{k})^{\frac{p}{2k}} \geq 0
\end{align}
for all $p \geq p_0$.

Assume for every $0 < \vep < 1$ and all $\delta \in (0, 1)$ we know
\begin{align*}
D_{p - 2k}(\delta)^{p - 2k} \leq C_{p - 2k, \vep}q^{a(p - 2k, p_0)}\delta^{-(\frac{p - 2k}{2} - \frac{k(k + 1)}{2}) - c(p_0)(1 - \frac{1}{k})^{\frac{p- 2k}{2k}} - \vep}
\end{align*}
for some $p \in p_0 + 2k\N$ (this is true for $p = p_0 + 2k$) and $C_{p - 2k, \vep}$ is allowed to depend on $C_1$. Then \eqref{eq:D_iterate} gives
\begin{align*}
D_{p}(\delta)^{p} &\leq C(\log \delta^{-1})^{3p}\Big[D_{p}(\delta^{1 - \vep})^{p}\\
&\quad+C_{p - 2k, \vep}q^{a(p, p_0)}\delta^{-(k^2 + 4k - 2)\vep}\delta^{-\frac{1}{k}(\frac{p}{2} + \frac{k(k - 3)}{2})}\delta^{-(1 - \frac{1}{k})(\frac{p - 2k}{2} - \frac{k(k + 1)}{2}) - c(p_0)(1 - \frac{1}{k})^{\frac{p - 2k}{2k} + 1} - \vep}\Big]\\
&= C(\log \delta^{-1})^{3p}D_{p}(\delta^{1 - \vep})^{p} + CC_{p - 2k, \vep}q^{a(p, p_0)}\delta^{-(\frac{p}{2} - \frac{k(k + 1)}{2}) - c(p_0)(1 - \frac{1}{k})^{\frac{p}{2k}} - (k^2 + 4k)\vep}
\end{align*}
for all $\delta, \vep \in (0, 1)$ where $C$ here depends only on $k$ and $p$.
Iterating this inequality $M$ times with $M$ to be chosen later gives that
\begin{align*}
D_{p}(\delta)^{p} &\leq C^{M}(\log \delta^{-1})^{3Mp}D_{p}(\delta^{(1 - \vep)^{M}})^{p}\nonumber\\
&\quad\quad\quad + CC_{p - 2k, \vep}q^{a(p, p_0)}\delta^{-(k^2 + 4k)\vep}\sum_{j = 0}^{M - 1}C^{j}(\log \delta^{-1})^{3pj}\delta^{-(1 - \vep)^{j}[(\frac{p}{2} - \frac{k(k + 1)}{2}) + c(p_0)(1 - \frac{1}{k})^{\frac{p}{2k}}]}.
\end{align*}
Trivially, we have $D_{p}(\delta^{(1 - \vep)^{M}}) \leq\delta^{-(1 - \vep)^{M}/2}$. Thus
\begin{align}\label{iterationM}
\begin{aligned}
D_{p}(\delta)^{p} &\leq C^{M}(\log \delta^{-1})^{3Mp}\delta^{-(1 - \vep)^{M}p/2}\\
&+ CC_{p - 2k, \vep}q^{a(p, p_0)}\delta^{-(k^2 + 4k)\vep}\sum_{j = 0}^{M - 1}C^{j}(\log \delta^{-1})^{3pj}\delta^{-(1 - \vep)^{j}[(\frac{p}{2} - \frac{k(k + 1)}{2}) + c(p_0)(1 - \frac{1}{k})^{\frac{p}{2k}}]}.
\end{aligned}
\end{align}
By \eqref{pineq}, the power of $\delta^{-1}$ in \eqref{iterationM}
is positive and so using that $(1 - \vep)^{j} \leq 1$, the sum can
be controlled by
$$M C^{M}(\log \delta^{-1})^{3Mp}\delta^{-(\frac{p}{2} - \frac{k(k + 1)}{2}) - c(p_0)(1 - \frac{1}{k})^{\frac{p}{2k}}}.$$
Inserting this into \eqref{iterationM} and choosing
$M$ be the least integer such that $(1 - \vep)^{M} \leq \vep$ (and so $M = \lceil \frac{\log \vep^{-1}}{\log (1 - \vep)^{-1}} \rceil$) then shows that
\begin{align*}
D_{p}(\delta) \lsm_{p, \vep, C_{1}} q^{a(p, p_0)/p} 
(\log \delta^{-1})^{3M} \delta^{-(\frac{1}{2} - \frac{k(k + 1)}{2p}) - \frac{c(p_0)}{p}(1 - \frac{1}{k})^{\frac{p}{2k}} - \frac{(k^2 + 4k) \vep}{p}}
\end{align*}
for all $\delta, \vep \in (0, 1)$. Since $(\log \delta^{-1})^{3M} \lesssim_{\vep} \delta^{-\vep}$, by redefining $\vep$ we have 
\begin{align*}
D_{p}(\delta) \lsm_{p, \vep, C_{1}} q^{a(p, p_0)/p}  \delta^{-(\frac{1}{2} - \frac{k(k + 1)}{2p}) - \frac{c(p_0)}{p}(1 - \frac{1}{k})^{\frac{p}{2k}} - \vep}.
\end{align*}

\appendix

\section{Proof of $\mathfrak{D}_{2k}(\delta) \lsm_{\vep} \delta^{-\vep}$}

Fix $k \in \N$ and a prime $q > k$. For $\delta \in q^{-\N}$, let $S(\delta)$ be the smallest constant such that the reverse square function estimate
\[
\int_{\F^k} |g|^{2k} \leq S(\delta)^{2k} \int_{\F^k} (\sum_{K \in P_{\delta}} |g_K|^2)^{k}
\]
holds for every Schwartz function $g$ on $\F^k$ with Fourier transform supported in $\bigcup_{K \in P_{\delta}} \theta_K$. We will prove that 
\[
S(\delta) \lesssim_{\vep} \delta^{-\vep}
\]
for every $\vep > 0$, which by Minkowski's inequality is stronger than the assertion $\mathfrak{D}_{2k}(\delta) \lesssim_{\vep} \delta^{-\vep}$.

Let $\delta \in q^{-\N}$, $g$ be as above, and $\kappa \in q^{-\N} \cap [\delta,1]$. The broad/narrow dichotomy given by the pointwise estimate \eqref{eq:g_pointwise} implies
\begin{equation} \label{eq:BG}
\int_{\F^k} |g|^{2k} \leq 2^{2k-1} k^{2k} \sum_{I \in P_{\kappa}} \int_{\F^k} |g_I|^{2k} + 2^{2k-1} \kappa^{-(4k-2)} \sum_{\substack{I_1, \dots, I_k \in P_{\kappa} \\ d(I_i,I_j) > \kappa \, \forall i \ne j}} \int_{\F^k} |g_{I_1} \dots g_{I_k}|^2
\end{equation}
Furthermore, by a rescaling argument similar to that in Lemma~\ref{rescale}, we have
\begin{align}
\sum_{I \in P_{\kappa}} \int_{\F^k}  |g_I|^{2k}
&\leq S(\frac{\delta}{\kappa})^{2k} \sum_{I \in P_{\kappa}}
\int_{\F^k} (\sum_{K \in P_{\delta}(I)} |g_K|^2)^k \leq S(\frac{\delta}{\kappa})^{2k} \int_{\F^k} (\sum_{K \in P_{\delta}} |g_K|^2)^k \label{eq:narrow}
\end{align}
where we used the pointwise inequality $\sum_{I \in P_{\kappa}}  (\sum_{K \in P_{\delta}(I)} |g_K|^2)^k \leq  (\sum_{K \in P_{\delta}} |g_K|^2)^k$ in the last inequality.
To proceed further, fix now $I_1, \dots, I_k \in P_{\kappa}$ with $d(I_i,I_j) > \kappa$ for all $i \ne j$. We expand
\[
\int_{\F^k} |g_{I_1} \dots g_{I_k}|^2 
= \sum_{\substack{K_i \in P_{\delta}(I_i) \\ i=1,\dots,k}} \sum_{\substack{\bar{K}_j \in P_{\delta}(I_j) \\ j=1,\dots,k}} \int_{\F^k}g_{K_1} \dots g_{K_k} \overline{g_{\bar{K}_1}} \dots \overline{g_{\bar{K}_k}}
\]
and write
\[
\int_{\F^k}g_{K_1} \dots g_{K_k} \overline{g_{\bar{K}_1}} \dots \overline{g_{\bar{K}_k}}
=  [\widehat{g_{K_1}}*\dots*\widehat{g_{K_k}}*\gkbbh](0).
\]
For each $\bar{K}_1 \in P_{\delta}(I_1), \ldots, \bar{K}_k \in P_{\delta}(I_k)$, we count the number of ordered $k$-tuples $(K_1,\dots,K_k)$ with $K_i \in P_{\delta}(I_i)$ for $i = 1,\dots,k$ and $0 \in \supp(\widehat{g_{K_1}}*\dots*\widehat{g_{K_k}}*\gkbbh)$. The proof of Lemma~\ref{lem:counting} shows that the number of such ordered $k$-tuples is $\leq (q \kappa)^{-k(k-1)}$ (in fact, here we only need that $\widehat{g_{K_j}}$ is supported in the cube $\tau_{K_j}$ rather than the smaller parallelepiped $\theta_{K_j}$). So using Cauchy-Schwarz,
\begin{align*}
\sum_{\substack{K_i \in P_{\delta}(I_i) \\ i=1,\dots,k}} \sum_{\substack{\bar{K}_j \in P_{\delta}(I_j) \\ j=1,\dots,k}} \int_{\F^k}g_{K_1} \dots g_{K_k} \overline{g_{\bar{K}_1}} \dots \overline{g_{\bar{K}_k}} 
&\leq (q \kappa)^{-k(k-1)} \sum_{\substack{K_i \in P_{\delta}(I_i) \\ i=1,\dots,k}} \int_{\F^k} |g_{K_1} \dots g_{K_k}|^2.
\end{align*}
It follows that 
\begin{align}
\sum_{\substack{I_1, \dots, I_k \in P_{\kappa} \\ d(I_i,I_j) > \kappa \, \forall i \ne j}} \int_{\F^k} |g_{I_1} \dots g_{I_k}|^2
&\leq (q \kappa)^{-k(k-1)} \int_{\F^k} \Big(\sum_{K \in P_{\delta}} |g_K|^2 \Big)^k.
\label{eq:broad}
\end{align}
Alternatively, multilinear restriction estimate and $L^2$ orthogonality says that for any ball $B_{\delta^{-1}}$ of radius $\delta^{-1}$ in $\F^k$, one has
\[
\int_{B_{\delta^{-1}}} |g_{I_1} \dots g_{I_k}|^2 \lesssim_{\kappa} (\delta^{k-1})^k \prod_{j=1}^k \int_{B_{\delta^{-1}}} |g_{I_j}|^2 = |B_{\delta^{-1}}|^{-(k-1)} \prod_{j=1}^k \int_{B_{\delta^{-1}}} \Big(\sum_{K_j \in P_{\delta}(I_j)} |g_{K_j}|^2\Big),
\]
and since each $|g_{K_j}|$ is constant on $B_{\delta}^{-1}$, we have 
\[
|B_{\delta^{-1}}|^{-(k-1)} \prod_{j=1}^k \int_{B_{\delta^{-1}}} \Big(\sum_{K_j \in P_{\delta}(I_j)} |g_{K_j}|^2\Big) = \int_{B_{\delta^{-1}}} \prod_{j=1}^k  \Big(\sum_{K_j \in P_{\delta}(I_j)} |g_{K_j}|^2\Big).
\]
Summing over all $B_{\delta^{-1}} \subset \F^k$ and all $I_1, \dots, I_k \in P_{\kappa}$, we have
\[
\sum_{\substack{I_1, \dots, I_k \in P_{\kappa} \\ d(I_i,I_j) > \kappa \, \forall i \ne j}} \int_{\F^k} |g_{I_1} \dots g_{I_k}|^2 \lesssim_{\kappa} \int_{\F^k} \Big(\sum_{K \in P_{\delta}} |g_K|^2 \Big)^k,
\]
which for the purposes below is as good as \eqref{eq:broad}.
Putting \eqref{eq:narrow} and \eqref{eq:broad} back into \eqref{eq:BG}, we have
\[
S(\delta)^{2k} \leq 2^{2k-1} k^{2k} S(\frac{\delta}{\kappa})^{2k} + 2^{2k-1} \kappa^{-(4k-2)} (q \kappa)^{-k(k-1)}.
\]
Iterating this gives
\[
S(\delta)^{2k} \leq (2^{2k-1} k^{2k})^{N} S(\frac{\delta}{\kappa^N})^{2k} + N 2^{2k-1} \kappa^{-(4k-2)} (q \kappa)^{-k(k-1)}
\]
for all positive integers $N$ for which $\kappa^N \geq \delta$; in particular, applying this with $N = \lfloor \frac{ \log \delta^{-1} }{\log \kappa^{-1}} \rfloor$, and noting that $S(\delta/\kappa^N) \leq (\delta/\kappa^N)^{-1/2} \leq \kappa^{-1/2}$, we have
\[
S(\delta)^{2k} \leq \delta^{-\frac{\log (2^{2k-1} k^{2k})}{\log \kappa^{-1}}} \kappa^{-k} + \frac{\log \delta^{-1}}{\log \kappa^{-1}} 2^{2k-1} \kappa^{-(4k-2)} (q \kappa)^{-k(k-1)}.
\]
By choosing $\kappa = \kappa(\vep)$ sufficiently small so that $\frac{\log (2^{2k-1} k^{2k})}{\log \kappa^{-1}} \leq 2k \vep$, one obtains $S(\delta) \lesssim_{\vep} \delta^{-\vep}$, as desired.

\bibliographystyle{amsplain}
\bibliography{oldVMVT}

\providecommand{\bysame}{\leavevmode\hbox to3em{\hrulefill}\thinspace}
\providecommand{\MR}{\relax\ifhmode\unskip\space\fi MR }
% \MRhref is called by the amsart/book/proc definition of \MR.
\providecommand{\MRhref}[2]{%
  \href{http://www.ams.org/mathscinet-getitem?mr=#1}{#2}
}
\providecommand{\href}[2]{#2}
\begin{thebibliography}{10}

\bibitem{Biggs19}
Kirsti~D. Biggs, \emph{Efficient congruencing in ellipsephic sets: the
  quadratic case}, Acta Arith. \textbf{200} (2021), no.~4, 331--348.

\bibitem{BBH}
Kirsti~D. Biggs, Julia Brandes, and Kevin Hughes, \emph{{R}einforcing a
  {P}hilosophy: {A} counting approach to square functions over local fields},
  arXiv:2201.09649.

\bibitem{BDG}
Jean Bourgain, Ciprian Demeter, and Larry Guth, \emph{Proof of the main
  conjecture in {V}inogradov's mean value theorem for degrees higher than
  three}, Ann. of Math. (2) \textbf{184} (2016), no.~2, 633--682.

\bibitem{BG}
Jean Bourgain and Larry Guth, \emph{Bounds on oscillatory integral operators
  based on multilinear estimates}, Geom. Funct. Anal. \textbf{21} (2011),
  no.~6, 1239--1295.

\bibitem{CDGJLM}
Alan Chang, Jaume de~Dios~Pont, Rachel Greenfeld, Asgar Jamneshan, Zane~Kun Li,
  and Jos\'e Madrid, \emph{Decoupling for fractal subsets of the parabola},
  Mathematische Zeitschrift \textbf{301} (2022), 1851--1879.

\bibitem{Cordoba77}
Antonio C\'{o}rdoba, \emph{The {K}akeya maximal function and the spherical
  summation multipliers}, Amer. J. Math. \textbf{99} (1977), no.~1, 1--22.

\bibitem{Cordoba82}
\bysame, \emph{Geometric {F}ourier analysis}, Ann. Inst. Fourier (Grenoble)
  \textbf{32} (1982), no.~3, vii, 215--226.

\bibitem{Demeter-book}
Ciprian Demeter, \emph{Fourier restriction, decoupling, and applications},
  Cambridge Studies in Advanced Mathematics, vol. 184, Cambridge University
  Press, Cambridge, 2020.

\bibitem{Drury}
Stephen~W. Drury, \emph{Restrictions of {F}ourier transforms to curves}, Ann.
  Inst. Fourier (Grenoble) \textbf{35} (1985), no.~1, 117--123.

\bibitem{Fefferman73}
Charles Fefferman, \emph{A note on spherical summation multipliers}, Israel J.
  Math. \textbf{15} (1973), 44--52.

\bibitem{F02-zeta}
Kevin Ford, \emph{Vinogradov's integral and bounds for the {R}iemann zeta
  function}, Proc. London Math. Soc. (3) \textbf{85} (2002), no.~3, 565--633.

\bibitem{F02-zerofree}
\bysame, \emph{Zero-free regions for the {R}iemann zeta function}, Number
  theory for the millennium, {II} ({U}rbana, {IL}, 2000), A K Peters, Natick,
  MA, 2002, pp.~25--56.

\bibitem{GGPRY}
Philip~T. Gressman, Shaoming Guo, Lillian~B. Pierce, Joris Roos, and Po-Lam
  Yung, \emph{Reversing a philosophy: from counting to square functions and
  decoupling}, J. Geom. Anal. \textbf{31} (2021), no.~7, 7075--7095.

\bibitem{GLY21}
Shaoming Guo, Zane~Kun Li, and Po-Lam Yung, \emph{Improved discrete restriction
  for the parabola}, arXiv:2103.09795, to appear in \emph{Mathematical Research
  Letters}.

\bibitem{GLY19}
\bysame, \emph{A bilinear proof of decoupling for the cubic moment curve},
  Trans. Amer. Math. Soc. \textbf{374} (2021), no.~8, 5405--5432.

\bibitem{GLYZ}
Shaoming Guo, Zane~Kun Li, Po-Lam Yung, and Pavel Zorin-Kranich, \emph{A short
  proof of $\ell^2$ decoupling for the moment curve}, American J. Math.
  \textbf{143} (2021), no.~6, 1983--1998.

\bibitem{Guth-325}
Larry Guth, \emph{A restriction estimate using polynomial partitioning}, J.
  Amer. Math. Soc. \textbf{29} (2016), no.~2, 371--413.

\bibitem{GMW}
Larry Guth, Dominique Maldague, and Hong Wang, \emph{{I}mproved decoupling for
  the parabola}, arXiv:2009.07953, to appear in the \emph{Journal of the
  European Mathematical Society}.

\bibitem{HB17}
D.~R. Heath-Brown, \emph{A new {$k$}th derivative estimate for exponential sums
  via {V}inogradov's mean value}, Proceedings of the Steklov Institute of
  Mathematics \textbf{296} (2017), 88--103.

\bibitem{HB15}
\bysame, \emph{{T}he cubic case of {V}inogradov's mean value theorem -- a
  simplified approach to {W}ooley's ``efficient congruencing"}, Essential
  Number Theory \textbf{1} (2022), no.~1, 1--12.

\bibitem{HW}
Jonathan Hickman and James Wright, \emph{{A} non-archimedean variant of
  {L}ittlewood--{P}aley theory for curves}, The Journal of Geometric Analysis
  \textbf{33} (2023), no.~104.

\bibitem{Karatsuba}
A.~A. Karatsuba, \emph{Mean value of the modulus of a trigonometric sum}, Izv.
  Akad. Nauk SSSR Ser. Mat. \textbf{37} (1973), 1203--1227.

\bibitem{Li18}
Zane~Kun Li, \emph{An {$l^2$} decoupling interpretation of efficient
  congruencing: the parabola}, Rev. Mat. Iberoam. \textbf{37} (2021), no.~5,
  1761--1802.

\bibitem{L43}
U.~V. Linnik, \emph{On {W}eyl's sums}, Rec. Math. [Mat. Sbornik] N.S.
  \textbf{12(54)} (1943), 28--39.

\bibitem{AM}
Akshat Mudgal, \emph{{D}iameter free estimates for the quadratic {V}inogradov
  mean value theorem}, Proceedings of the London Mathematical Society
  \textbf{126} (2023), no.~1, 76--128.

\bibitem{Pierce19}
Lillian~B. Pierce, \emph{{T}he {V}inogradov mean value theorem [after {W}ooley,
  and {B}ourgain, {D}emeter and {G}uth]}, Ast\'{e}risque Expos\'{e}s Bourbaki
  \textbf{407} (2019), 479--564.

\bibitem{Stechkin}
S.~B. Ste\v{c}kin, \emph{Mean values of the modulus of a trigonometric sum},
  Trudy Mat. Inst. Steklov. \textbf{134} (1975), 283--309.

\bibitem{Taibleson}
M.~H. Taibleson, \emph{Fourier analysis on local fields}, Princeton University
  Press, Princeton, N.J.; University of Tokyo Press, Tokyo, 1975.

\bibitem{Tao254A}
Terence Tao, \emph{254{A}, {N}otes 5: {B}ounding exponential sums and the zeta
  function},
  \url{https://terrytao.wordpress.com/2015/02/07/254a-notes-5-bounding-exponential-sums-and-the-zeta-function/}.

\bibitem{Tao-Recent}
\bysame, \emph{{R}ecent progress on the restriction conjecture},
  arXiv:math/0311181.

\bibitem{V}
R.~C. Vaughan, \emph{The {H}ardy-{L}ittlewood method}, second ed., Cambridge
  Tracts in Mathematics, vol. 125, Cambridge University Press, Cambridge, 1997.

\bibitem{Vinogradov1935}
I.M. Vinogradov, \emph{{N}ew estimates for {W}eyl sums}, Dokl. Akad. Nauk SSSR
  \textbf{8} (1935), 195--198.

\bibitem{VVZ}
V.~S. Vladimirov, I.~V. Volovich, and E.~I. Zelenov, \emph{{$p$}-adic analysis
  and mathematical physics}, Series on Soviet and East European Mathematics,
  vol.~1, World Scientific Publishing Co., Inc., River Edge, NJ, 1994.

\bibitem{WooleyICM}
Trevor~D. Wooley, \emph{Translation invariance, exponential sums, and
  {W}aring's problem}, Proceedings of the {I}nternational {C}ongress of
  {M}athematicians---{S}eoul 2014. {V}ol. {II}, Kyung Moon Sa, Seoul, 2014,
  pp.~505--529.

\bibitem{W16}
\bysame, \emph{The cubic case of the main conjecture in {V}inogradov's mean
  value theorem}, Adv. Math. \textbf{294} (2016), 532--561.

\bibitem{W19}
\bysame, \emph{Nested efficient congruencing and relatives of {V}inogradov's
  mean value theorem}, Proceedings of the London Mathematical Society
  \textbf{118} (2019), no.~4, 942--1016.

\end{thebibliography}
\end{document}